\documentclass[10pt,reqno]{amsart}
\usepackage[ams]{optional}

\usepackage{dsfont,amsmath,amsfonts,amscd,amssymb,graphicx,mathrsfs,eufrak}
\usepackage[all,arc,curve,color,frame]{xy}
\usepackage[usenames]{color}

\usepackage{setspace}
\usepackage{array,multirow,booktabs,longtable}

\opt{ams}{
\setcounter{tocdepth}{1}
\newcommand{\tocspace}{0.1ex}
\let\oldtocsection=\tocsection
\let\oldtocsubsection=\tocsubsection
\let\oldtocsubsubsection=\tocsubsubsection
\renewcommand{\tocsection}[3]{\hspace{0em}\oldtocsection{#1}{#2}{#3}\vspace{\tocspace}}
\renewcommand{\tocsubsection}[3]{ \hspace{1em} \oldtocsubsection{#1}{\small{#2}}{\small{#3}}\vspace{\tocspace} }
\renewcommand{\tocsubsubsection}[3]{\hspace{2em}\oldtocsubsubsection{#1}{\small{#2}}{\small{#3}}}
}

\newcommand{\marginparstretch}{0.6}
\let\oldmarginpar\marginpar
\renewcommand\marginpar[1]{\-\oldmarginpar[\framebox{\setstretch{\marginparstretch}\begin{minipage}{\marginparwidth}{\raggedleft\tiny #1}\end{minipage}}]{\framebox{\setstretch{\marginparstretch}\begin{minipage}{\marginparwidth}{\raggedright\tiny #1}\end{minipage}}}}

\usepackage[colorlinks]{hyperref}
\usepackage{tikz,mathrsfs}
\usetikzlibrary{arrows,decorations.pathmorphing,decorations.pathreplacing,positioning,shapes.geometric,shapes.misc,decorations.markings,decorations.fractals,calc,patterns}

\tikzset{
        cvertex/.style={circle,draw=black,inner sep=1pt,outer sep=3pt},
        vertex/.style={circle,fill=black,inner sep=1pt,outer sep=3pt},
        star/.style={circle,fill=yellow,inner sep=0.75pt,outer sep=0.75pt},
        tvertex/.style={inner sep=1pt,font=\scriptsize},
        gap/.style={inner sep=0.5pt,fill=white}}

\tikzstyle{mybox} = [draw=black, fill=blue!10, very thick,
    rectangle, rounded corners, inner sep=10pt, inner ysep=20pt]
\tikzstyle{boxtitle} =[fill=blue!50, text=white,rectangle,rounded corners]

\newtheorem{thm}{Theorem}[section]
\newtheorem{prop}[thm]{Proposition}
\newtheorem{lemma}[thm]{Lemma}
\newtheorem{defin}[thm]{Definition}
\newtheorem{cor}[thm]{Corollary}
\newtheorem{summary}[thm]{Summary}

\opt{ams}{\newtheorem*{bigassum}{Assumptions}}
\opt{lms}{}

\opt{ams}{\theoremstyle{definition} }

\newtheorem{example}[thm]{Example}
\newtheorem{setup}[thm]{Setup}

\newtheorem{remark}[thm]{Remark}

\opt{ams}{\newtheorem*{bigquestion}{Question}}
\newtheorem{notation}[thm]{Notation}
\opt{lms}{\newtheorem{bigquestion}{Question}}

\numberwithin{equation}{section}

\newcounter{enumeratenoindentcounter}

\newcommand{\m}{\mathfrak{m}}
\newcommand{\n}{\mathfrak{n}}
\newcommand{\p}{\mathfrak{p}}

\renewcommand{\t}[1]{\textnormal{#1}}

\def\op{\mathop{\rm op}\nolimits}
\def\rest{\mathop{\rm res}\nolimits}

\def\mod{\mathop{\rm mod}\nolimits}

\def\coh{\mathop{\rm coh}\nolimits}
\def\Qcoh{\mathop{\rm Qcoh}\nolimits}
\def\Mod{\mathop{\rm Mod}\nolimits}

\def\Hom{\mathop{\rm Hom}\nolimits}

\def\RHom{\mathop{\rm {\bf R}Hom}\nolimits}

\def\End{\mathop{\rm End}\nolimits}
\def\Ext{\mathop{\rm Ext}\nolimits}

\def\add{\mathop{\rm add}\nolimits}

\def\Ker{\mathop{\rm Ker}\nolimits}

\def\Supp{\mathop{\rm Supp}\nolimits}

\def\Spec{\mathop{\rm Spec}\nolimits}

\def\Perf{\mathop{\rm{per}}\nolimits}

\def\D{\mathop{\rm{D}^{}}\nolimits}

\def\Db{\mathop{\rm{D}^b}\nolimits}

\def\Id{\mathop{\rm{Id}}\nolimits}
\newcommand{\idem}[1]{e_{#1}}
\newcommand{\K}{\mathop{{}_{}\mathbb{C}}\nolimits}

\newcommand{\con}{\mathrm{con}}
\newcommand{\fib}{\mathrm{fib}}
\newcommand{\CA}{\mathrm{A}_{\con}}
\newcommand{\CAR}{\mathrm{A}_{\mathrm{fib}}}

\newcommand{\AB}{\mathrm{A}}
\newcommand{\BB}{\mathrm{B}}

\def\ab{\mathop{\rm ab}\nolimits}
\def\redu{\mathop{\rm red}\nolimits}

\def\uotimes{\mathop{\underline{\otimes}}\nolimits}

\def\Rf{{\rm\bf R}f}

\def\RHom{{\rm{\bf R}Hom}}
\def\sHom{\mathcal{H}om}

\newcommand\RDerived[1]{{\rm\bf R}{#1}}

\newcommand\LDerived[1]{{\rm\bf L}{#1}}

\newcommand\art{\mathsf{Art}}
\newcommand\proart{\mathsf{pArt}}
\newcommand\cart{\mathsf{CArt}}
\newcommand\alg{\mathsf{Alg}}
\newcommand\calg{\mathsf{CAlg}}
\newcommand\Sets{\mathsf{Sets}}
\newcommand\cDef{c\mathcal{D}ef}
\newcommand\Def{\mathcal{D}ef}

\newcommand\DG{\mathsf{DG}}

\newcommand{\cA}{\mathcal{A}}
\newcommand{\cB}{\mathcal{B}}

\newcommand{\cF}{\mathcal{F}}
\newcommand{\cG}{\mathcal{G}}

\newcommand{\cJ}{\mathcal{J}}
\newcommand{\cK}{\mathcal{K}}
\newcommand{\cL}{\mathcal{L}}
\newcommand{\cM}{\mathcal{M}}
\newcommand{\cN}{\mathcal{N}}
\newcommand{\cO}{\mathcal{O}}

\newcommand{\cS}{\mathcal{S}}

\newcommand{\cV}{\mathcal{V}}
\newcommand{\cW}{\mathcal{W}}

\newcommand{\cY}{\mathcal{Y}}
\newcommand{\cZ}{\mathcal{Z}}

\newcommand\clocCon{f} 

\newlength\tempWidth
\newcommand\InSpaceOf[2]{
   \settowidth{\tempWidth}{$#1$}
   \phantom{#1}
   \hspace{-\tempWidth}
   {#2}}

\begin{document}
\opt{lms}{\title{Contractions and deformations}
\author{Will Donovan and Michael Wemyss}}
\opt{ams}{
\title{\textsc{Contractions and Deformations}}
\author{Will Donovan}
\address{Will Donovan, Yau Mathematical Sciences Center, Tsinghua University, Haidian District, Beijing 100084, China.}
\email{donovan@mail.tsinghua.edu.cn}
\author{Michael Wemyss}
\address{Michael Wemyss, The Mathematics and Statistics Building, University of Glasgow, University Place, Glasgow, G12 8SQ, UK.}
\email{michael.wemyss@glasgow.ac.uk}
\thanks{The first author was supported by World Premier International Research Center Initiative (WPI), MEXT, Japan, and by EPSRC grant~EP/G007632/1. The second author was supported by EPSRC grant~EP/K021400/1.}}
\opt{lms}{
\classno{Primary 14D15; Secondary 14E30, 14F05, 16E45, 16S38}
\extraline{The first author was supported by World Premier International Research Center Initiative (WPI), MEXT, Japan, and by EPSRC grant~EP/G007632/1. The second author was supported by EPSRC grant~EP/K021400/1.}
}
\maketitle

\begin{abstract}
Suppose that $f$ is a projective birational morphism with at most one-dimensional fibres between $d$-dimensional varieties $X$ and $Y$, satisfying $\Rf_*\cO_X=\cO_Y$. Consider the locus $L$ in $Y$ over which $f$ is not an isomorphism.  Taking the scheme-theoretic fibre $C$ over any closed point of $L$, we construct algebras $\AB_\fib$ and $\CA$ which prorepresent the functors of commutative deformations of $C$, and noncommutative deformations of the reduced fibre, respectively. Our main theorem is that the algebras $\CA$ recover $L$, and in general the commutative deformations of neither $C$ nor the reduced fibre can do this. As the $d=3$ special case, this proves the following contraction theorem: in a neighbourhood of the point, the morphism $f$ contracts a curve without contracting a divisor if and only if the functor of noncommutative deformations of the reduced fibre is representable. 
\end{abstract}

\parindent 20pt
\parskip 0pt

\tableofcontents

\section{Introduction}

Our setting is a contraction $f\colon X\to X_{\con}$ with at most one-dimensional fibres between $d$-dimensional varieties, satisfying $\Rf_*\cO_X=\cO_{X_{\con}}$. Writing $L \subset X_{\con}$ for the locus over which $f$ is not an isomorphism, it is a fundamental problem to characterise $L$, locally around a closed point in the base.   For this, it is natural to study the deformations of the curve(s) above the point, and the question which we answer in this paper is the following.

\begin{bigquestion}
Which deformation-theoretic framework detects the non-isomorphism locus~$L$, Zariski locally around a closed point $p\in X_{\con}$?
\end{bigquestion}

The answer turns out to lie in noncommutative deformations, without assumptions on the singularities of $X$, and in arbitrary dimension.  This process associates a noncommutative algebra to each point, which should be viewed as an invariant of the contraction~$f$. When $d=3$ and the algebra is finite-dimensional, its dimension has a curve-counting interpretation~\cite{TodaGV}, but in all cases the algebra structure gives information about the neighbourhood of the point. This extra information has applications in constructing derived autoequivalences \cite{DW1,DW3,HomMMP,Kawamata}, and also in the minimal model program, allowing us to control iteration of flops and count minimal models \cite{HomMMP}, and produce the first explicit examples of Type~$E$ flops \cite{BrownWemyss,Karmazyn}. It is also conjectured that such algebras classify smooth $3$-fold flops~\cite{DW1,HT}, and furthermore it is expected that they control divisor-to-curve contractions.

\begin{figure}
\[
\begin{array}{ccc}
\begin{array}{c}
\begin{tikzpicture}
\node at (0,0) {\begin{tikzpicture}[scale=1]
\coordinate (T) at (1.9,2);
\coordinate (TM) at (2.12-0.02,1.5-0.1);
\coordinate (BM) at (2.12-0.05,1.5+0.1);
\coordinate (B) at (2.1,1);
\draw[line width=0.5pt]\opt{colordiag}{[red]} (T) to [bend left=25] (TM);
\draw[line width=0.5pt]\opt{colordiag}{[red]} (BM) to [bend left=25] (B);

\foreach \y in {0.1,0.2,...,1}{ 
\draw[very thin]\opt{colordiag}{[blue]} ($(T)+(\y,0)+(0.02,0)$) to [bend left=25] ($(B)+(\y,0)+(0.02,0)$);
\draw[very thin]\opt{colordiag}{[blue]} ($(T)+(-\y,0)+(-0.02,0)$) to [bend left=25] ($(B)+(-\y,0)+(-0.02,0)$);;}
\draw[rounded corners=15pt,line width=0.5pt] (0.5,0) -- (1.5,0.3)-- (3.6,0) -- (4.3,1.5)-- (4,3.2) -- (2.5,2.7) -- (0.2,3) -- (-0.2,2)-- cycle;
\end{tikzpicture}};
\node at (0,-3.5) {\begin{tikzpicture}[scale=1]
\draw (1.1,0.75) -- (3.1,0.75);
\filldraw\opt{colordiag}{[red]} (2.1,0.75) circle (1pt);
\node at (2.5,0.6) {$\scriptstyle p\phantom{=L}$};
\node at (3,0.6) {$\scriptstyle \phantom{p=}L$};

\draw[rounded corners=12pt,line width=0.5pt] (0.5,0) -- (1.5,0.15)-- (3.6,0) -- (4.3,0.75)-- (4,1.6) -- (2.5,1.35) -- (0.2,1.5) -- (-0.2,0.6)-- cycle;
\end{tikzpicture}};
\draw[->] (0,-1.6) -- node[left] {$\scriptstyle f$} (0,-2.65);
\node at (-2.8,-3.6) {$X_{\con}$};
\node at (-2.8,-0.1) {$X$};
\end{tikzpicture}
\end{array}
& \qquad &
\begin{array}{c}
\begin{tikzpicture}
\node at (0,0) {\begin{tikzpicture}[scale=1]
\coordinate (T) at (1.9,2);
\coordinate (TM) at (2.12-0.02,1.5-0.1);
\coordinate (BM) at (2.12-0.05,1.5+0.1);
\coordinate (B) at (2.1,1);
\draw[line width=0.5pt]\opt{colordiag}{[red]} (T) to [bend left=25] (TM);
\draw[line width=0.5pt]\opt{colordiag}{[red]} (BM) to [bend left=25] (B);

\draw[rounded corners=15pt,line width=0.5pt] (0.5,0) -- (1.5,0.3)-- (3.6,0) -- (4.3,1.5)-- (4,3.2) -- (2.5,2.7) -- (0.2,3) -- (-0.2,2)-- cycle;
\end{tikzpicture}};
\node at (0,-3.5) {\begin{tikzpicture}[scale=1]
\filldraw\opt{colordiag}{[red]} (2.1,0.75) circle (1pt);
\node at (2.5,0.6) {$\scriptstyle p=L$};
\draw[rounded corners=12pt,line width=0.5pt] (0.5,0) -- (1.5,0.15)-- (3.6,0) -- (4.3,0.75)-- (4,1.6) -- (2.5,1.35) -- (0.2,1.5) -- (-0.2,0.6)-- cycle;
\end{tikzpicture}};
\draw[->] (0,-1.6) -- node[left] {$\scriptstyle f$} (0,-2.65);
\end{tikzpicture}
\end{array}
\end{array}
\]
\vspace{-1.5em}
\caption{Contractions of $3$-folds: divisor to curve, and curve to point.}
\label{fig.deformations}
\end{figure}
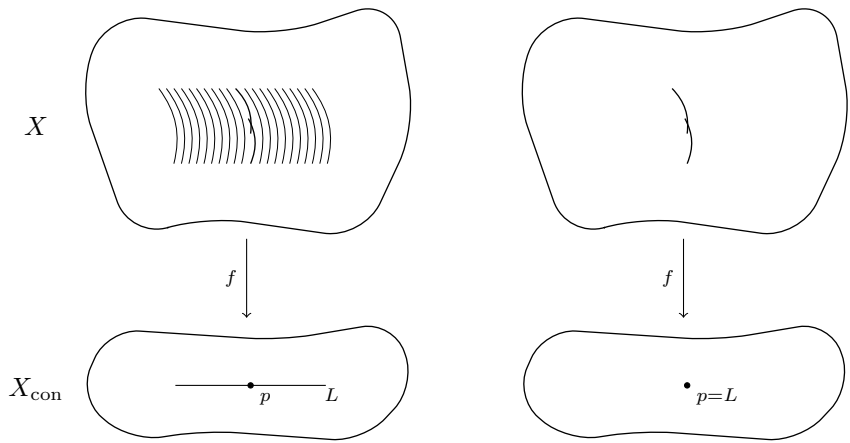

\subsection{Summary of Results} For a closed point $p\in L$, consider the scheme-theoretic fibre $C=f^{-1}(p)$.  Set-theoretically, it is well known that $C$ is a union of $\mathbb{P}^1$s, and we denote these by $C_1,\hdots,C_n$. To specify a deformation problem requires us to provide test objects, and to say what object(s) are being deformed. The test objects  for the noncommutative deformation functor are the category $\art_n$ of artinian augmented $\K^n$-algebras, and the objects being deformed are $\{ \cO_{C_1}(-1), \hdots, \cO_{C_n}(-1) \}$.  Informally, we wish to control the deformations of the $\cO_{C_i}(-1)$, and also the mutual extensions between them. Formally, as explained in~\S\ref{section defs}, this is encoded via the Maurer--Cartan formulation as a functor 
\[
\Def^{\cJ}\colon \art_n \rightarrow \Sets.
\] 

Given a closed point $p \in X_{\con}$, it is well-known that the formal fibre over $p$ is derived equivalent to a certain noncommutative ring~$\AB$. By taking suitable factors, as in \cite[\S2]{DW1} we obtain the contraction algebra $\CA$, referring the reader to \ref{def basic algebra} for full details. Our first main result is then the following. The special case where the locus~$L$ is a single point, $d=3$, and $n=1$, was previously shown in \cite[3.1]{DW1}.

\begin{thm}[=\ref{All defs prorep}]\label{defs prorep intro } For each closed point $p \in L$,  the algebra $\CA$ (depending on~$p$) prorepresents the functor of noncommutative deformations $\Def^{\cJ}$ of the reduced fibre over~$p$.
\end{thm}

It turns out that the geometry of the locus $L$ is controlled by the support of $\CA$.

\begin{thm}\label{main combo intro}
With the setup above, pick an affine open neighbourhood $\Spec R$ in $X_{\con}$, and consider $L_R:=L\cap \Spec R$. 
\begin{enumerate}
\item\textnormal{(=\ref{contraction algebra def 1}, \ref{contract1}\eqref{contract1 1})} There is an $R$-algebra $\Lambda_{\con}$ which satisfies $\Supp_R\Lambda_{\con}=L_R$.
\item\textnormal{(=\ref{ME lemma})} For each closed point $p \in L_R$, the completion of $\Lambda_{\con}$ at $p$ is morita equivalent to $\CA$.
\end{enumerate}
\end{thm}
This theorem has two main consequences.

\begin{cor}
The dimension of $L_R$ at $p$ is $\dim_{\mathfrak{R}}\Supp_{\mathfrak{R}}\CA$, where $\mathfrak{R}$ denotes the completion of $R$ at $p$. 
\end{cor}

\begin{thm}[={\ref{contraction theorem}}, contraction theorem]
When $d=3$, there is a neighbourhood of~$p$ over which $f$ does not contract a divisor if and only if $\dim_{\K}\CA<\infty$.
\end{thm}

The `only if' direction is easy and is known from our previous work \cite[2.13]{DW1}. The content is the `if' direction, and this requires significantly more technology.

\subsection{Comparing deformation theories}

We next show that other natural deformation functors do not control the geometry of $L$. As above, consider the scheme-theoretic fibre $C=f^{-1}(p)$ and the reduced curves $C_1,\hdots,C_n$ therein. To this data, we associate three other deformation problems. Again the details are left to \S\ref{section defs}, but the following table summarises all four functors, giving the test objects and the deformed object in each case. Here $\cart_n$ is the category of commutative artinian augmented $\K^n$-algebras.

\begin{equation*}
	\begin{tabular}{l*6c}
\toprule
\multirow{2}{*}{\bf Deformation problem}&&\multirow{2}{*}{\bf Functor}&&{\bf Test}&&{\bf Object(s)}\\
{\bf }&&{\bf }&&{\bf objects}&&{\bf deformed}\\
\midrule
Classical scheme-theoretic && $\cDef^{\cO_C}$&& $\cart_1$&& $\cO_C$ \\
Noncommutative scheme-theoretic && $\phantom{c}\Def^{\cO_C}$ && $\phantom{\mathsf{C}}\art_1$&& $\cO_C$\\
\cmidrule(l){1-7}
Commutative multi-pointed  && $\cDef^{\cJ}$&& $\cart_{n}$&& $\oplus_i \cO_{C_i}(-1)$ \\
Noncommutative multi-pointed  && $\phantom{c}\Def^{\cJ}$&& $\phantom{\mathsf{C}}\art_n$&& $\{\cO_{C_i}(-1)\}_i$  \\
\bottomrule\\
\end{tabular}
\end{equation*}

The following result drops out of our general construction.  It is quite surprising, since it says that noncommutative deformations of the scheme-theoretic fibre give nothing in addition to the classical ones.

\begin{thm}[=\ref{oCdefmain}]
The functors $\cDef^{\cO_C}$ and $\Def^{\cO_C}$ are prorepresented by the same object $\CAR$.
\end{thm}

In general, however, the prorepresenting objects for the functors $\cDef^{\cO_C}$, $\cDef^{\cJ}$ and $\Def^{\cJ}$ are different. We prove the following.

\begin{prop}[=\ref{trihedral1}, \ref{trihedral2}]
Neither $\cDef^{\cO_C}$ nor $\cDef^{\cJ}$ detect the dimension of the non-isomorphism locus $L$.
\end{prop}

This is clear when the fibre above $p$ has more than one irreducible curve, but is much more surprising when there is only a single irreducible curve in the fibres.  We produce in~\ref{trihedral2} a contraction, sketched below, with a one-dimensional non-isomorphism locus~$L$ in which the central point $0$ is $cD_4$, and all other points are $cA_1$.  The commutative deformations of all reduced fibres except the central one are infinite dimensional, whereas the noncommutative deformations are always infinite dimensional.
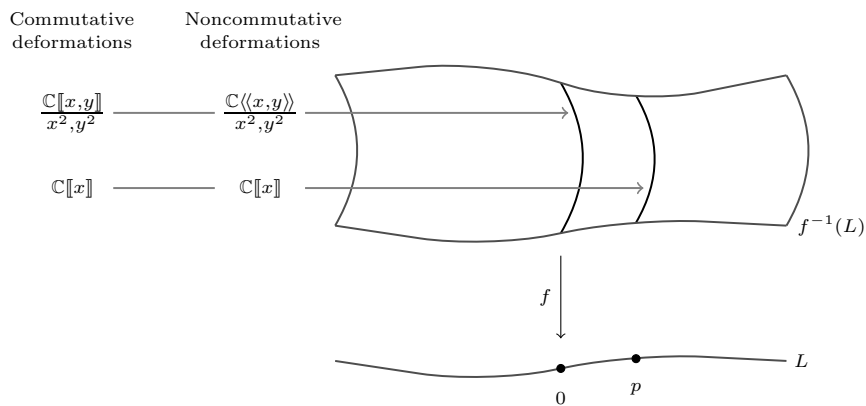
\begin{figure}[!h]
\def\shift{-0.3}
\[
\begin{tikzpicture}
\draw [bend right,line width=0.75pt]\opt{colordiag}{[red]} (3,-2.1) to (3,-0.1); 
\draw [bend right,line width=0.75pt]\opt{colordiag}{[red]} (4,-1.96) to (4,-0.27); 
\draw[color=black!70,rounded corners=25pt,line width=0.75pt] (0,0)-- (2,0.2) -- (4,-0.4) -- (6,0) ;
\draw[color=black!70,rounded corners=25pt,line width=0.75pt] (0,-2) -- (2,-2.3) -- (4,-1.9)  -- (6,-2);
\draw[color=black!70,bend left,line width=0.75pt] (0,0) to (0,-2);
\draw[color=black!70,bend left,line width=0.75pt] (6,0) to (6,-2);
\draw[color=black!70,rounded corners=25pt,line width=0.75pt] (0,-3.5+\shift) -- (2,-3.8+\shift) -- (4,-3.4+\shift)  -- (6,-3.5+\shift);
\draw[->] (3,-2.4) -- node[left] {$\scriptstyle f$} (3,-3.5);
\filldraw\opt{colordiag}{[red]} (3,-3.6+\shift) circle (1.5pt);
\filldraw\opt{colordiag}{[red]} (4,-3.47+\shift) circle (1.5pt);
\node at (3,-4+\shift) {$\scriptstyle 0$};
\node at (4,-3.87+\shift) {$\scriptstyle p$};
\node at (-1,-0.5) {$\frac{\K\langle\!\langle x,y\rangle\!\rangle}{x^2,y^2}$};
\node at (-1,-1.5) {$\scriptstyle \K[\![ x]\!]$};
\node at (-3.5,-0.5) {$\frac{\K[\![ x,y]\!]}{x^2,y^2}$};
\node at (-3.5,-1.5) {$\scriptstyle \K[\![ x]\!]$};
\node[align=center, text width=2cm] at (-1,0.75) {\scriptsize Noncommutative};
\node[align=center, text width=2cm]  at (-3.5,0.75) {\scriptsize Commutative};
\node[align=center, text width=2cm] at (-1,0.45) {\scriptsize deformations};
\node[align=center, text width=2cm]  at (-3.5,0.45) {\scriptsize deformations};
\draw[->,line width=0.75,black!50] (-0.4,-0.5) to (3.1,-0.5);
\draw[->,line width=0.75,black!50] (-0.4,-1.5) to (4.1,-1.5);
\draw[line width=0.75,black!50] (-2.95,-0.5) to (-1.6,-0.5);
\draw[line width=0.75,black!50] (-2.95,-1.5) to (-1.6,-1.5);
\node at (6.6,-2) {$\scriptstyle f^{-1}(L)$};
\node at (6.2,-3.5+\shift) {$\scriptstyle L$};
\end{tikzpicture}
\]
\vspace{-3em}
\caption{Commutative versus noncommutative deformations of $C^{\redu}.$}
\label{fig.comm vs nc}
\end{figure}

Thus, it follows that $\Def^{\cJ}$ is the functor that controls the contractibility of curves.  For the convenience of the reader, we summarise the above together with the main results of \cite{DW3} in the following table.  To discuss autoequivalences, we further assume that $f$ is a flopping contraction, and $X$ is $\mathds{Q}$-factorial with only Gorenstein terminal singularities.
\vspace{1em}
\begin{equation*}
\begin{tabular}{l*6c}
\toprule
\multirow{2}{*}{\bf Deformation problem}&&\multirow{2}{*}{\bf Functor}&&{\bf Detects}&&{\bf Corresponds to}\\
{\bf }&&{\bf }&&{\bf divisor?}&&{\bf autoequivalence?}\\\midrule
Classical && $\cDef^{\cO_C}$ && No&& Yes\\
\cmidrule(l){1-7}
Commutative multi-pointed && $\cDef^{\cJ}$&&No && No\\
Noncommutative multi-pointed &&  $\phantom{c}\Def^{\cJ}$&& Yes&& Yes\\
\bottomrule\\
\end{tabular}
\end{equation*}

Our method to prove the above theorems involves noncommutative deformation theory associated to DGAs.  By passing through various derived equivalences and embeddings, we reduce the geometric deformation problem into an easier problem about simultaneously  deforming a collection of simple modules on a complete local ring obtained by tilting.

\subsection{Structure of the Paper}
In \S\ref{naivesection}, we recall the naive noncommutative deformation theory developed by Laudal~\cite{Laudal}, and subsequently Eriksen~\cite{Eriksen}. We then describe DG deformation theory in \S\ref{DGdefsection}, noting that it is equivalent by work of Segal~\cite{Segal} and Efimov--Lunts--Orlov~\cite{ELO1}, and establish tools for comparing DG deformations which will arise in our construction. 

In \S\ref{Section3} we give the geometric setup, before proving the prorepresentability results for noncommutative deformations of the reduced fibre. The key observation, building on \S\ref{section defs}, is that this deformation functor is isomorphic to a DG deformation functor associated to a specific locally free resolution.  In \S\ref{section5} we use these results to prove \ref{defs prorep intro }, \ref{main combo intro}, and the Contraction Theorem.

In \S\ref{section:schemefibre} we prove the prorepresentability results for both commutative and noncommutative deformations of the scheme-theoretic fibre, and show that they are prorepresented by the same object. The fact that $\CA$ and $\CAR$ are obtained as factors of a common ring~$\AB$ then allows us to relate the different deformation functors, and we do this in \S\ref{subsect:comparison}.  We conclude in \S\ref{examples section} by giving examples which illustrate the necessity of noncommutative deformations in the above theorems.

\subsection{Conventions}\label{section conventions}  Throughout we work over the field of complex numbers $\K$. Unqualified uses of the word `module' refer to left modules, and $\mod A$ denotes the category of finitely generated left $A$-modules.  For two $\K$-algebras $A$ and $B$, an $A$-$B$ bimodule is the same thing as an $A\otimes_{\K}B^{\op}$-module. We use the functorial convention for composing arrows, so $f\cdot g$ means $f$ then $g$.  In particular, naturally this makes $M\in \mod A$ into an $\End_R(M)$-module.  We remark that these conventions are opposite to those in \cite{DW1, DW3}, but we do this to match the conventions in \cite{Eriksen, ELO1}. Similarly, for quivers, DG category morphisms, and matrix multiplication, we write $ab$ for $a$ then $b$. We reserve the notation $f \circ g$ for the composition $g$~then~$f$.

For an abelian category $\cA$, given $a\in\cA$, we let $\add (a)$ denote all possible summands of finite sums of $a$. Given furthermore $b\in\cA$ where $a$ is a summand of $b$, we write $[a]$ for the two-sided ideal of $\End(b)$ consisting of all morphisms that factor through a member of $\add (a)$.

\medskip
\noindent
\textbf{Acknowledgements.}  The authors would like to thank Jon Pridham, Ed Segal, Olaf Schn\"urer  and Yukinobu Toda for helpful discussions relating to this work, and anonymous referees for useful comments and suggestions.

\section{Naive and DG Deformations}
\label{section defs}

This section is mainly a review of known material, and is used to set notation.  Noncommutative deformations of modules were introduced by Laudal \cite{Laudal}, and we review these naive deformation functors in \S\ref{naivesection} below.  In our geometric setting later,  this naive definition is necessary in order to establish the prorepresenting object in \ref{All defs prorep}. 

However, it is cumbersome to compare two or more naive deformation functors, as is demonstrated in the proofs of \cite[2.6, 2.8, 2.19]{DW1}, and for this the setting of multi-pointed DG deformation functors is much better suited.  We review this theory in \S\ref{DGdefsection}, being a slight generalisation of the setting of \cite{Segal,ELO1}, before in \S\ref{axiomsection} proving some general DG deformation functor results.

\subsection{Naive Deformations of Modules}\label{naivesection}
From the geometric motivation of the introduction, where we want to deform $n$ reduced curves in a fibre simultaneously, the test objects for the naive deformation functors are objects of the category $\art_n$, defined as follows.

\begin{defin}\label{defin.n-pointed_non-comm_K-alg} An \emph{$n$-pointed $\K$-algebra} $\Gamma$ is an associative $\K$-algebra, together with $\K$-algebra morphisms $p\colon \Gamma \to \K^n$ and $i: \K^n \to \Gamma$ such that $ip = \Id$.  A morphism of $n$-pointed $\K$-algebras $\psi\colon (\Gamma,p,i)\to(\Gamma',p',i')$ is an algebra homomorphism $\psi\colon\Gamma\to\Gamma'$ such that
\[
\begin{tikzpicture}
\node (A) at (0,0) {$\K^n$};
\node (B1) at (1.5,0.75) {$\Gamma$};
\node (B2) at (1.5,-0.75) {$\Gamma'$};
\node (C) at (3,0) {$\K^n$};
\draw[->] (A) -- node[above] {$\scriptstyle i$} (B1);
\draw[->] (A) -- node[below] {$\scriptstyle i'$} (B2);
\draw[->] (B1) -- node[above] {$\scriptstyle p$} (C);
\draw[->] (B2) -- node[below] {$\scriptstyle p'$} (C);
\draw[->] (B1) -- node[right] {$\scriptstyle \psi$}(B2);
\end{tikzpicture}
\]
commutes.  We denote the category of $n$-pointed $\K$-algebras by $\alg_n$.  We denote the full subcategory consisting of those objects that are commutative rings by $\calg_n$.

We write $\art_n$ for the full subcategory of $\alg_n$ consisting of objects $(\Gamma,p,i)$ for which $\dim_{\K}\Gamma<\infty$ and the augmentation ideal $\n:=\Ker p$ is nilpotent.  We write $\cart_n$ for the full subcategory of $\art_n$ consisting of those objects that are commutative rings.
\end{defin}

\begin{remark}\label{full simples}
If $\Gamma$ is an associative ring, then by \cite[1.3]{Eriksen2}, there exists $p$, $i$ such that $(\Gamma,p,i)\in\art_n$ if and only if $\Gamma$ is an artinian $\mathbb{C}$-algebra with precisely $n$ simple modules (up to isomorphism), each of them one-dimensional over $\mathbb{C}$.
\end{remark}

\begin{notation}\label{remark.matrix_viewpoint}If $(\Gamma,p,i)\in\art_n$, the structure morphisms $p$ and $i$ allow us to lift the canonical idempotents $\{\idem{1}, \ldots, \idem{n}\}$ of $\K^n$ to $\Gamma$. We will write 
\[
\Gamma_{ij} := \idem{i} \Gamma \idem{j},
\]
and consider $\Gamma$ as a matrix ring $(\Gamma_{ij})$ under standard matrix multiplication. Accordingly, a left $\Gamma$-module $M$ may be described in terms of its vector of summands $M_i := e_i M $, and a right $\Gamma$-module $N$ can be described by its summands $N_i := N e_i $.
\end{notation}

\begin{defin}\label{def:def functor}
Given a $\K$-algebra $\Lambda$, choose a family $\cS=\{S_1,\hdots,S_n\}$ of objects in $\Mod\Lambda$. The deformation functor
\[
\Def_{\Lambda}^{\cS}\colon\art_n\to\Sets
\]
is defined by sending
\[
(\Gamma,\n)\mapsto
\left. \left \{ (M,\underline{\delta})
\left|\begin{array}{l}M\in \Mod\Lambda\otimes_{\mathbb{C}}\Gamma^{\op}\\ 
M \cong (S_i\otimes_{\mathbb{C}}\Gamma_{ij})\mbox{ as right }\Gamma\mbox{-modules}\\ \underline{\delta}=(\delta_i), \mbox{where each }\delta_i\colon M\otimes_\Gamma (\Gamma /\n)e_i\xrightarrow{\sim} S_i\end{array}\right. \right\} \middle/ \sim \right.
\]
where 
\begin{enumerate}
\item $(S_i\otimes_{\mathbb{C}}\Gamma_{ij})$ refers to the $\mathbb{C}$-vector space $
(S_i\otimes_{\mathbb{C}}\Gamma_{ij}):=\bigoplus_{i,j=1}^n(S_i\otimes_{\mathbb{C}}\Gamma_{ij})$,
with the natural right $\Gamma$-module structure coming from the multiplication in $\Gamma$.

\item $(M,\underline{\delta})\sim (N,\underline{\delta}')$ iff there exists an isomorphism $\tau\colon M\to N$ of bimodules such that  the following diagram commutes for all $i=1,\hdots,n$.
\[
\begin{tikzpicture}
\node (a1) at (0,0) {$M\otimes_\Gamma (\Gamma/\n) e_i$};
\node (a2) at (3,0) {$N\otimes_\Gamma (\Gamma/\n) e_i$};
\node (b) at (1.5,-1) {$S_i$};
\draw[->] (a1) -- node[above] {$\scriptstyle \tau \otimes 1 $} (a2);
\draw[->] (a1) -- node[gap] {$\scriptstyle \delta_i$} (b);
\draw[->] (a2) -- node[gap] {$\scriptstyle \delta'_i$} (b);
\end{tikzpicture}
\]
\vspace{-2em}
\end{enumerate}
\end{defin}

An important problem in deformation theory is determining when deformation functors are prorepresentable, and also effectively describing the prorepresenting object.  We briefly recall these notions, mainly to fix notation.

For any $(\Gamma,p,i)\in\alg_n$, setting $I(\Gamma):=\Ker p$ we consider the $I(\Gamma)$-adic completion $\widehat{\Gamma}$ of $\Gamma$, defined by
\[
\widehat{\Gamma}:=\varprojlim \Gamma/I(\Gamma)^n.
\]
The canonical morphism $\psi_\Gamma\colon \Gamma\to\widehat{\Gamma}$ belongs to $\alg_n$.  We say that $\Gamma\in\alg_n$ is \emph{complete} if $\psi_\Gamma$ is an isomorphism.  The pro-category $\proart_n$ is then defined to be the full subcategory of $\alg_n$ consisting of those objects $(\Gamma,p,i)$ for which $\Gamma$ is $I(\Gamma)$-adically complete, and $\Gamma/I(\Gamma)^r\in\art_n$ for all $r\geq 1$.  It is clear that $\art_n\subseteq \proart_n$.

For a deformation functor $F\colon\art_n\to\Sets$, recall that 
\begin{enumerate}
\item $F$ is called \emph{prorepresentable} if $F\cong\Hom_{\proart_n}(\Gamma,-)|_{\art_n}$ for some $\Gamma\in\proart_n$.
\item $F$ is called \emph{representable} if $F\cong\Hom_{\art_n}(\Gamma,-)$ for some $\Gamma\in\art_n$.
\end{enumerate}
It is clear that if $F$ is prorepresented by $\Gamma\in\proart_n$, then $F$ is representable if and only if $\dim_{\mathbb{C}}\Gamma<\infty$.

It is well-known that the functor in \ref{def:def functor} is prorepresentable if $\Ext^t_\Lambda(\bigoplus S_i,\bigoplus S_i)$ is finite dimensional for $t=1,2$ \cite{Laudal}: we will not use this fact below, however, instead preferring to establish the prorepresenting object in a much more direct way.

\subsection{DG Deformations}\label{DGdefsection}
With our conventions as in the introduction, recall that a DG~category is a graded category $\mathsf{A}$ whose morphism spaces are endowed with a differential~$d$, i.e.\ homogeneous maps of degree one satisfying $d^2=0$, such that 
\[
d(gf)=g(df)+(-1)^t(dg)f
\]
for all $g\in\Hom_{\mathsf{A}}(a,b)$ and all $f\in\Hom_{\mathsf{A}}(b,c)_t$ for $t \in \mathbb{Z}$.  In this paper we will be interested in the category $\DG_n$, which has as objects those DG categories with precisely $n$ objects.  If $\mathsf{A},\mathsf{B}\in\DG_n$, recall that a DG functor $F\colon\mathsf{A}\to\mathsf{B}$ is a graded functor  such that $F(df)=d(Ff)$ for all morphisms $f$.  A \emph{quasi-isomorphism} $F\colon\mathsf{A}\to\mathsf{B}$ is a DG functor inducing a bijection on objects, and quasi-isomorphisms $\Hom_{\mathsf{A}}(a_1,a_2)\to\Hom_{\mathsf{B}}(Fa_1,Fa_2)$ for all $a_1,a_2\in\mathsf{A}$. Two categories $\mathsf{A},\mathsf{B}\in\DG_n$ are called \emph{quasi-isomorphic} if they are connected through a finite, non-directed chain of quasi-isomorphisms.

\begin{notation}
Suppose that $\mathsf{A}\in\DG_n$, and $(\Gamma,\n)\in\art_n$, and recall from \ref{remark.matrix_viewpoint} that $\n_{ij}:=e_i\n\hspace{0.1em}e_j$.  Define $\mathsf{A}\uotimes\n:=\bigoplus_{i,j=1}^n(\mathsf{A}\uotimes\n)_{ij}$, where
\[
(\mathsf{A}\uotimes\n)_{ij}:=\Hom_{\mathsf{A}}(i,j)\otimes_{\K} \n_{ij}.
\]
Observe that $\mathsf{A}\uotimes\n$ has the natural structure of an object in $\DG_n$ (but with no units) with differential $d(a\otimes x):=d(a)\otimes x$. Thus we may consider $\mathsf{A}\uotimes\n$ as a DGLA, with bracket
\[
\quad[a\otimes x,b\otimes y]:=ab\otimes xy-(-1)^{\deg(a)\deg(b)}ba\otimes yx
\] 
for homogeneous $a,b\in\mathsf{A}$.
\end{notation}
Since $(\Gamma,\n)\in\art_n$, by definition $\n^r=0$ for some $r\geq 1$, hence $\mathsf{A}\uotimes\n$ is a nilpotent DGLA.  This being the case, we can consider the standard Maurer--Cartan formulation to obtain a deformation functor. Here $\mathsf{A}^t\underline{\otimes}\,\n$ denotes a homogeneous piece of $\mathsf{A}\,\underline{\otimes}\,\n$.  
\begin{defin}
Given $(\mathsf{A},d)\in\DG_n$,  the associated DG deformation functor
\[
\Def^{\mathsf{A}}\colon\art_n\to\Sets
\]
is defined by sending
\[
(\Gamma,\n)\mapsto
\left. \left \{ 
\xi\in \mathsf{A}^1\underline{\otimes}\,\n
\left|
\,
d(\xi)+\frac{1}{2}[\xi,\xi]=0
\right. \right\} \middle/ \sim \right.
\]
where as usual the equivalence relation $\sim$ is induced by the gauge action.  Explicitly, two elements $\xi_1,\xi_2\in \mathsf{A}^1\uotimes\n$ are said to be gauge equivalent if there exists $x\in\mathsf{A}^0\uotimes\n$ such that
\[
\xi_2=e^x*\xi_1:=\xi_1+\sum_{j=0}^{\infty}\frac{([x,-])^j}{(j+1)!}([x,\xi_1]-d(x)).
\]
\end{defin}

The following is a mild extension of the well-known $n=1$ case.  The proof is very similar to the known $n=1$ proofs (see e.g.\ \cite[8.1]{ELO1}, \cite[3.2]{Manetti}, \cite[2.4]{GM}), so we do not give it here.
\begin{thm}\label{Qis give def iso}
Suppose that $\mathsf{A},\mathsf{B}\in\DG_n$ are quasi-isomorphic. Then the deformation functors $\Def^\mathsf{A}$ and $\Def^\mathsf{B}$ are isomorphic.
\end{thm}

\subsection{Basic Results}\label{axiomsection}
Controlling noncommutative deformations of curves in the next section requires the following two preliminary results, and a corollary.  All are well-known in the case $n=1$.

To fix notation, suppose that $\cA$ is an abelian category with enough injectives and that $x,y$ are two chain complexes with objects from $\cA$.  Let $\Hom^{\DG}_{\cA}(x,y)$ be the DG~$\K$-module with
\[
\Hom^{\DG}_{\cA}(x,y)_t:=\{ (f_s)_{s\in\mathbb{Z}}\mid f_s\colon x_s\to y_{s+t}\}
\]
and differential $\delta\colon f\mapsto fd_y-(-1)^{\deg(f)}d_xf$.

Now choose a family of objects $a_1,\hdots,a_n\in\cA$, an injective resolution $0\to a_i\to I^i_\bullet$ for each $a_i$, and set $I:=\bigoplus_{i=1}^n I^i_\bullet$.   From this, we form $(\End_\cA^{\DG}(I),\delta)$, considered naturally as an object of $\DG_n$.

\begin{lemma}\label{FLlemma}
Suppose that $\cA, \cB$ are abelian categories, and $F\colon \cA\to\cB$ is an additive functor with left adjoint $L$.  Choose a family of objects $a_1,\hdots,a_n\in\cA$, and for each choose an injective resolution $0\to a_i\to I^i_\bullet$.  If
\begin{enumerate}
\item\label{FLlemma 1} $L$ is exact,
\item\label{FLlemma 2} $\mathbf{R}^{t}F(a_i)=0$ for all $t>0$ and all $i=1,\hdots,n$,
\item\label{FLlemma 3} The counit $L \circ F \to \Id$ is an isomorphism on each object $a_i$,
\end{enumerate}
then $\End^{\DG}_{\cA}(\bigoplus I^i_\bullet)$ and $\End^{\DG}_{\cB}(\bigoplus FI^i_\bullet)$ are quasi-isomorphic in $\DG_n$.
\end{lemma}

\begin{proof} Consider the obvious map
\begin{eqnarray*}
F\colon \End^{\DG}_{\cA}\left(\bigoplus I^i_\bullet\right)\to\End^{\DG}_{\cB}\left(\bigoplus FI^i_\bullet\right).\label{the obvious q}
\end{eqnarray*}
To show that this is a quasi-isomorphism it suffices, by adjunction, to show that 
\begin{eqnarray}
\End^{\DG}_{\cA}\left(\bigoplus I^i_\bullet\right)\to\Hom^{\DG}_{\cA}\left(\bigoplus LFI^i_\bullet, \bigoplus I^i_\bullet\right)\label{the obvious q 2}
\end{eqnarray}
given by composing with the counit morphisms $L F I^i_\bullet \to I^i_\bullet$ is a quasi-isomorphism. Consider the following commutative square, given by applying the counit $L \circ F \to \Id$ to the resolution of $a_i$.
\[
\begin{tikzpicture}[xscale=1.1]
\node (a1) at (2,0) {$L F a_i$};
\node (a2) at (4,0) {$a_i$};
\node (b1) at (2,-1) {$L F I^i_\bullet$};
\node (b2) at (4,-1) {$I^i_\bullet$};
\draw[->] (a1)--node[above]{$\scriptstyle \sim$}(a2);
\draw[->] (b1)--(b2);
\draw[->] (a1)--node[sloped,left,anchor=south]{$\scriptstyle\sim$}(b1);
\draw[->] (a2)--node[sloped,left,anchor=south]{$\scriptstyle\sim$}(b2);
\end{tikzpicture}
\]
To see that the left-hand arrow is a quasi-isomorphism, note that condition \eqref{FLlemma 1} implies that $F$ preserves injective objects, and \eqref{FLlemma 2} implies that $F$ preserves the injective resolutions of the $a_i$, hence $0\to Fa_i\to FI^i_\bullet$ are injective resolutions: the claim then follows because $L$ is exact. The top arrow is a quasi-isomorphism by condition \eqref{FLlemma 3}, and so we deduce that the bottom arrow is a quasi-isomorphism. It follows that \eqref{the obvious q 2} is a quasi-isomorphism because~$I^i_\bullet$ is $h$-injective, giving the lemma.
\end{proof}

\begin{remark}We are grateful to an anonymous referee for suggesting an improvement to a previous argument for this lemma.\end{remark}

Keeping the notation as above, for each of the objects $a_1,\hdots,a_n\in\cA$, choose a left resolution $Q^\bullet_i\to a_i\to 0$, where for now the $Q$'s are arbitrary.   Set $Q:=\bigoplus_{i=1}^n Q^\bullet_i$, and consider
\[
\Delta:=
\begin{pmatrix}
\End^{\DG}_{\cA}(Q[1])&\Hom^{\DG}_{\cA}(Q[1],I)\\[1mm]
0& \End^{\DG}_{\cA}(I)
\end{pmatrix}.
\]
This can be viewed as an object in $\DG_n$ in the obvious way: the homomorphism space between object $i$ and object $j$ is 
\[
\begin{pmatrix}
\End^{\DG}_{\cA}(Q_i^\bullet[1])&\Hom^{\DG}_{\cA}(Q_i^\bullet[1],I^j_\bullet)\\[1mm]
0& \End^{\DG}_{\cA}(I^j_\bullet)
\end{pmatrix},
\]
with differential as in \cite[\S8]{ELO1}. There are natural projections $p_1\colon\Delta\to\End_{\cA}^{\DG}(Q[1])$ and $p_2\colon\Delta\to\End_{\cA}^{\DG}(I)$ in $\DG_n$, and splicing the left resolutions with the right resolutions gives an exact complex $Q[1]\to I$.

\begin{prop}[Keller]\label{Kellerplus}
Suppose that $\cA$ is an abelian category, and choose a family of objects $a_1,\hdots,a_n\in\cA$.  With notation as above, 
\begin{enumerate}
\item\label{Kellerplus 1} The projection $p_1$ is a quasi-isomorphism in $\DG_n$.
\item\label{Kellerplus 2} If $\Hom^{\DG}_\cA(Q[1],Q[1]\to I)$ is exact, then $p_2$ is a quasi-isomorphism in $\DG_n$.
\end{enumerate}
In particular, provided that $\Hom^{\DG}_\cA(Q[1],Q[1]\to I)$ is exact, $\End_{\cA}^{\DG}(Q)$ and $\End_{\cA}^{\DG}(I)$ are quasi-isomorphic in $\DG_n$.
\end{prop}
\begin{proof}
As above, the complex $Q[1]\to I$ is exact.\\
(1) By construction of the upper triangular $\Delta$, as in \cite[\S8]{ELO1} 
\[
\Ker p_1=\Hom^{\DG}_\cA(Q[1]\to I,I).
\]   
Since  $I$ is $h$-injective, it follows that $\Ker p_1$ is exact, and thus $p_1$ is a quasi-isomorphism.\\
(2) Again by construction, $\Ker p_2=\Hom^{\DG}_\cA(Q[1],Q[1]\to I)$, and so if by assumption this is exact, $p_2$ is a quasi-isomorphism.\\
The final statement follows from (1) and (2), since clearly $\End^{\DG}_\cA(Q)\cong\End^{\DG}_\cA(Q[1])$.
\end{proof}

The following is now a direct consequence of \cite[2.8]{Segal}, and says that the naive deformations and the DG deformations are the same when we deform distinct simples.
\begin{cor}\label{Ed Lemma}
Suppose that $\Lambda$ is a $\K$-algebra, and that $\cS=\{S_1,\hdots,S_n\}\subseteq\Mod\Lambda$ are simple and distinct.  Choose injective resolutions $0\to S_i\to I^i_\bullet$ and set $\mathsf{A}:=\End^{\DG}_{\Lambda}(I)\in\DG_n$ as above. Then $ \Def^{\mathsf{A}}\cong \Def^{\cS}_{\Lambda}$.
\end{cor}
\begin{proof}
Under the assumption that the $S_i$ are distinct simples, Segal \cite[2.6, 2.8]{Segal} shows that $\Def^{\cS}_{\Lambda}$ is isomorphic to the DG deformation functor associated to the bar resolutions of the simples.  Since projective resolutions are $h$-projective, the conditions of \ref{Kellerplus}\eqref{Kellerplus 2} are satisfied, so the bar resolution DGA is quasi-isomorphic in $\DG_n$ to $\mathsf{A}$.  The result then follows from \ref{Qis give def iso}.
\end{proof}

\section{Deformations of Reduced Fibres}\label{Section3}
In this section, in the setting of contractions with at most one-dimensional fibres, we show that the functor of simultaneous noncommutative deformations of the reduced fibre is prorepresented by a naturally defined algebra $\CA$.  This algebra is a factor of one obtained by tilting, and this extra control over the prorepresenting object allows us in \S\ref{section5} to prove that noncommutative deformations recover the contracted locus.

\subsection{Setup}
This subsection fixes notation.  Throughout the paper, we will refer to the three setups in \ref{setupglobal}, \ref{setupZariski} and \ref{setupcomplete} below.

\begin{setup}\label{setupglobal}
(Global) Suppose that $f\colon X\to X_{\con}$ is a projective birational morphism between noetherian integral normal $\mathbb{C}$-schemes, with $\Rf_*\cO_X=\cO_{X_{\con}}$, such that the fibres are at most one-dimensional.  Throughout, we write $L$ for the locus of (not necessarily closed) points of $X_{\con}$ above which $f$ is not an isomorphism.
\end{setup}

We make no assumptions on the singularities of $X$.  Next, for any closed point $p\in L$, we pick an affine neighbourhood $\Spec R$ in $X_{\con}$ containing $p$, and after base change consider the following Zariski local setup.

\begin{setup}\label{setupZariski}
(Zariski local) Suppose that $f\colon U\to\Spec R$  is a projective birational morphism between noetherian integral normal $\mathbb{C}$-schemes, with $\Rf_*\cO_U=\cO_R$, such that the fibres are at most one-dimensional. 
\end{setup}
In dimension $3$, an easy example is the blowup of $\mathbb{A}^3$ at the ideal $(x,y)$, but the setup also includes arbitrary flips and flops of multiple curves, as well as divisorial contractions to curves.  We make no assumptions on the singularities of $U$.  

With the assumptions in \ref{setupZariski}, it is well-known \cite[3.2.8]{VdB1d} that there is a bundle $ \cV:=\cO_U\oplus\cN$  inducing a derived equivalence
\begin{eqnarray}
\begin{array}{c}
\begin{tikzpicture}
\node (a1) at (0,0) {$\phantom{.}\Db(\coh U)$};
\node (a2) at (5,0) {$\Db(\mod \End_U(\cV)).$};
\draw[->] (a1) -- node[above] {$\scriptstyle\RHom_U(\cV,-)$} node [below] {$\scriptstyle\sim$} (a2);
\end{tikzpicture}
\end{array}\label{derived equivalence}
\end{eqnarray}
Throughout we set
\[
\Lambda:=\End_U(\cV) = \End_U(\cO_U\oplus\cN),
\]
and recall from the conventions in \S\ref{section conventions} that if $\cF,\cG\in\coh U$ where $\cF$ is a summand of~$\cG$, then we define the ideal $[\cF]$ to be the two-sided ideal of $\End_U(\cG)$ consisting of all morphisms factoring through $\add \cF$.   

The following is similar to \cite[2.12]{DW1}, but the definition is now more subtle since in general $\End_{U}(\cV)\ncong \End_R(f_*\cV)$, whereas there is such an isomorphism in the setting of~\cite{DW1}.  The upshot is that we must work on $U$, and not $\Spec R$.

\begin{defin}\label{contraction algebra def 1}
With notation as above, we define the \emph{contraction algebra} associated to~$\Lambda$ to be $\Lambda_{\con}:=\End_U(\cO_U\oplus\cN)/[\cO_U]$.
\end{defin}

The algebra $\Lambda_{\con}$ defined above depends on $\Lambda$ and thus the choice of derived equivalence \eqref{derived equivalence}, but this is accounted for in the formal fibre setting below, after passing through morita equivalences.  Also, we remark that since $\cN\notin\add \cO_U$ (else $f$ is an isomorphism, e.g.\ by \ref{notequiv lemma}), the contraction algebra $\Lambda_{\con}$ is necessarily non-zero.

To obtain well-defined invariants that do not depend on choices, and also to relate to the deformation theory in the following \S\ref{Sect def contraction}, we now pass to the formal fibre.
\begin{setup}\label{setupcomplete}
(Complete local) Suppose that $f\colon \mathfrak{U}\to\Spec \mathfrak{R}$  is a projective birational morphism between noetherian integral normal $\mathbb{C}$-schemes,  with $\Rf_*\cO_{\mathfrak{U}}=\cO_{\mathfrak{R}}$, where $\mathfrak{R}$~is complete local and  the fibres of $f$ are at most one-dimensional. 
\end{setup}

After passing to this formal fibre, the Zariski local derived equivalence has a particularly nice form, which we now briefly review.  Fix a closed point $\m \in L$, and write $\mathfrak{R}:=\widehat{R}$.   The above derived equivalence \eqref{derived equivalence} induces an equivalence
\[
\begin{array}{c}
\begin{tikzpicture}
\node (a1) at (0,0) {$\Db(\coh\mathfrak{U})$};
\node (a2) at (4,0) {$ \Db(\mod\widehat{\Lambda}),$};
\draw[->] (a1) -- node[above] {$\scriptstyle\RHom_{\mathfrak{U}}(\widehat{\cV},-)$} node [below] {$\scriptstyle\sim$} (a2);
\end{tikzpicture}
\end{array}
\]
and this can be described much more explicitly.  We let $C=\pi^{-1}(\m)$ where $\m$ is the unique closed point of $\Spec \mathfrak{R}$, then giving $C$ the reduced scheme structure, we can write $C^{\redu}=\bigcup _{i=1}^nC_i$ with each $C_i\cong\mathbb{P}^1$.  Let $\cL_i$ denote the line bundle on $\mathfrak{U}$ such that $\cL_i\cdot C_j=\delta_{ij}$. Recall that the multiplicity of $C_i$ in $C$ is given by the length of the local ring of $C$ at the generic point of $C_i$, considered as a module over the local ring of $\mathfrak{U}$ at the same point. If this is one, set $\cM_i:=\cL_i$, else define $\cM_i$ to be given by the maximal extension
\[
0\to\cO_{\mathfrak{U}}^{\oplus(r-1)}\to\cM_i\to\cL_i\to 0
\]
associated to a minimal set of $r-1$ generators of $H^1(\mathfrak{U},\cL_i^{*})$ as an $\mathfrak{R}$-module \cite[3.5.4]{VdB1d}.  Then $\cO_{\mathfrak{U}}\oplus \bigoplus_{i=1}^n\cM_i^*$
is a tilting bundle on $\mathfrak{U}$ \cite[3.5.5]{VdB1d}.   By  \cite[3.5.5]{VdB1d} we can write
\[
\cO_{\mathfrak{U}}\oplus\widehat{\cN}\cong \cO_{\mathfrak{U}}^{\oplus a_0}\oplus \bigoplus_{i=1}^n (\cM_i^*)^{\oplus a_i}
\]
for some $a_i\in\mathbb{N}$ and so consequently $\widehat{\Lambda}\cong\End_{\mathfrak{U}}(\cO_{\mathfrak{U}}^{\oplus a_0}\oplus \bigoplus_{i=1}^n (\cM_i^*)^{\oplus a_i})$.

\begin{defin}\label{def basic algebra}
We write $\cM^*:=\bigoplus_{i=1}^n \cM_i^*$ and define 
\[
\AB:=\End_{\mathfrak{U}}(\cO_{\mathfrak{U}}\oplus \cM^*),
\]
which is the basic algebra morita equivalent to $\widehat{\Lambda}$.  From this, we define the {\em contraction algebra associated to $\clocCon$} to be 
\[
\CA:=\End_{\mathfrak{U}}(\cO_{\mathfrak{U}}\oplus \cM^*)/[\cO_{\mathfrak{U}}]\cong \End_{\mathfrak{U}}(\cM^*)/[\cO_{\mathfrak{U}}],
\]
and we define the {\em fibre algebra associated to $\clocCon$} to be 
\[
\AB_{\fib}:=\End_{\mathfrak{U}}(\cO_{\mathfrak{U}}\oplus \cM^*)/[\cM^*]\cong \End_{\mathfrak{U}}(\cO_{\mathfrak{U}})/[\cM^*].
\]
\end{defin}

\begin{remark}\label{Afibcomm}
Since $\End_{\mathfrak{U}}(\cO_{\mathfrak{U}})\cong\mathfrak{R}$, it follows from the definition that $\AB_{\fib}$ is always commutative, although $\CA$ need not be.
\end{remark}

To establish various homological properties we will need to pass through morita equivalences between the algebras $\AB$ and $\widehat{\Lambda}$, and between the algebras $\CA$ and~$\widehat{\Lambda}_{\con}$.  Here we describe these equivalences, mainly to fix notation. In analogy with \cite[\S5.3]{DW1} throughout we write
\[
\cY:=\cO_{\mathfrak{U}}\oplus \bigoplus \cM_i^*, \qquad \cZ:=\cO_{\mathfrak{U}}^{\oplus a_0}\oplus \bigoplus (\cM_i^*)^{\oplus a_i},
\]
so that $\AB=\End_{\mathfrak{U}}(\cY)$ and $\widehat{\Lambda}=\End_{\mathfrak{U}}(\cZ)$.   Writing
\[
P:=\Hom_{\mathfrak{U}}(\cY,\cZ), \qquad Q:=\Hom_{\mathfrak{U}}(\cZ,\cY)
\]
it is clear that both $P$ and $Q$ have the structure of bimodules, namely ${}_{\widehat{\Lambda}}P_{\InSpaceOf{\widehat{\Lambda}}{\AB}}$ and ${}_{\InSpaceOf{\widehat{\Lambda}}{\AB}}Q_{\widehat{\Lambda}}$. It is easy to see that $P$ is a progenerator, and that this induces the following result.
\begin{lemma}\label{ME lemma}
With notation as above, there is a morita equivalence
\begin{eqnarray}
\begin{array}{c}
\begin{tikzpicture}[xscale=1]
\node (d1) at (3,0) {$\mod \AB$};
\node (e1) at (7.5,0) {${}_{}\mod\widehat{\Lambda},$};
\draw[->,transform canvas={yshift=+0.4ex}] (d1) to  node[above] {$\scriptstyle \mathbb{F}:=\Hom_{\AB}(P,-)=-\otimes_{\AB}Q $} (e1);
\draw[<-,transform canvas={yshift=-0.4ex}] (d1) to node [below]  {$\scriptstyle \Hom_{\widehat{\Lambda}}(Q,-)=-\otimes_{\widehat{\Lambda}}P $} (e1);
\end{tikzpicture}
\end{array}\label{ME1}
\end{eqnarray}
and a morita equivalence
\begin{eqnarray}
\begin{array}{c}
\begin{tikzpicture}[xscale=1]
\node (d1) at (3,0) {$\mod \CA$};
\node (e1) at (7.5,0) {${}_{}\mod\widehat{\Lambda}_{\con}.$};
\draw[->,transform canvas={yshift=+0.4ex}] (d1) to  node[above] {$\scriptstyle -\otimes_{\CA}\mathbb{F}\CA$} (e1);
\draw[<-,transform canvas={yshift=-0.4ex}] (d1) to node [below]  {$\scriptstyle \Hom_{\widehat{\Lambda}_{\con}}(\mathbb{F}\CA,-) $} (e1);
\end{tikzpicture}
\end{array}\label{ME2}
\end{eqnarray}
\end{lemma}

\subsection{Deformations and Contraction Algebras}\label{Sect def contraction}
In this subsection we will prove that the contraction algebra $\CA$ prorepresents various natural deformation functors. We will translate deformation problems across a variety of different functors, so we now set notation, as in \cite[(2.G)]{DW3}.

\begin{notation}\label{TandSnotation}
We pick a closed point $\m\in L$, and as above consider $C:=f^{-1}(\m)$.  We write $C^{\redu}=\bigcup _{i=1}^nC_i$ with each $C_i\cong\mathbb{P}^1$, and put $T_i$ for the simple $\Lambda$-modules corresponding to the (perverse) sheaves $\cO_{C_i}(-1)$ across the derived equivalence~\eqref{derived equivalence}.  Further, we denote the simple $\AB$-modules by $S_i:=\mathbb{F}^{-1}\widehat{T}_i$, so that
\[
\begin{array}{c}
\begin{tikzpicture}[xscale=1]
\node (d1) at (2,0.8) {$\Db(\Qcoh U)$};
\node (e1) at (6,0.8) {$\Db(\Mod \Lambda)$};
\node (g1) at (10,0.8) {$\Db(\Mod \AB)$};
\draw[->] (d1.5) to  node[above] {$\scriptstyle\RHom_U(\cV,-)$} (e1.175);
\draw[<-] (d1.-5) to node [below]  {$\scriptstyle-\otimes_{\Lambda}^{\bf L}\cV$} (e1.-175);
\draw[->] (e1.5) to  node[above] {$\scriptstyle \mathbb{F}^{-1}\circ\widehat{(-)}$} (g1.175);
\draw[<-] (e1.-5) to node [below]  {$\scriptstyle \rest\circ\,\mathbb{F}$} (g1.-175);
\node (d2) at (2,0) {$\cO_{C_i}(-1)$};
\draw[<->] (d1.east |- 3.5,0) -- (e1.west |- 5.5,0);
\node (e2) at (6,0) {$T_i$};
\draw[->] (e1.east |- 3.5,0) -- (g1.west |- 5.5,0);
\node (g2) at (10,0) {$S_i$};
\end{tikzpicture}
\end{array}
\] 
For the remainder of this subsection, we will use the following simplified notation:
\begin{enumerate}
\item $\Def_X$ for $\Def^{\mathsf{A}_1}$, where $\mathsf{A}_1\in\DG_n$ is obtained from the injective resolutions of the~$\cO_{C_i}(-1)\in\coh X$.
\item $\Def_U$ for $\Def^{\mathsf{A}_2}$, where $\mathsf{A}_2\in\DG_n$ is obtained from the injective resolutions of the~$\cO_{C_i}(-1)\in\coh U$.
\item $\Def_\Lambda$ for $\Def^{\mathsf{A}_3}$, where $\mathsf{A}_3\in\DG_n$ is obtained from the injective resolutions of the~$T_i\in\mod\Lambda$.
\item $\Def_\AB$ for $\Def^{\mathsf{A}_4}$, where $\mathsf{A}_4\in\DG_n$ is obtained from the injective resolutions of the~$S_i\in\mod\AB$.
\end{enumerate}
\end{notation}

The next result is now an easy corollary of the DG results in \S\ref{axiomsection}, vastly simplifying \cite[2.6, 2.8, 2.19]{DW1}.
\begin{thm}\label{All defs prorep}
With the global setup of \ref{setupglobal}, and notation in \ref{TandSnotation},
\begin{enumerate}
\item\label{All defs prorep 1} $\Def_X\cong \Def_U$.
\item\label{All defs prorep 2} $\Def_U\cong \Def_\Lambda$.
\item\label{All defs prorep 3} $\Def_\Lambda\cong\Def_{\AB}$.
\item\label{All defs prorep 4} $\Def_{\AB}\cong \Hom_{\proart_n}(\CA,-)$.
\end{enumerate}
In particular, all the deformation functors above are prorepresented by $\CA$.
\end{thm}
\begin{proof}
By \ref{Qis give def iso}, we just need to show that the DGAs are quasi-isomorphic.\\
(1) Consider $F:=i_*\colon \Qcoh U\to\Qcoh X$, with exact left adjoint $L=i^*$. Since $C_i$ is closed, $\RDerived i_* \cO_{C_i}(-1)=i_*\cO_{C_i}(-1)$ for each $i = 1,\hdots,n.$  Further, the counit $L \circ F \to \Id$ is an isomorphism for all sheaves, in particular for the sheaves $\cO_{C_i}(-1)$.  Hence the result follows from \ref{FLlemma}.\\
(2) Set $F:=\Hom_U(\cV,-)\colon \Qcoh U\to \Mod\Lambda$, with (non-exact) left adjoint $L:=\cV\otimes_\Lambda-$. To establish the claim we will first use \ref{Kellerplus}, so set $E_i:=\cO_{C_i}(-1)$, and for each $i=1,\hdots,n$ choose an injective resolution $0\to E_i\to I_\bullet^i$.  On the other hand, choose a projective resolution $P_i^\bullet\to T_i\to 0$ of each of the $T_i$.  Since the sheaf $E_i$ corresponds to the module $T_i$ across the equivalence \eqref{derived equivalence}, it follows that $\cV\otimes^{\bf L}_\Lambda T_i\cong E_i$. In particular all higher cohomology groups vanish, and so $LP_i^\bullet\to E_i\to 0$ is exact, giving a locally free resolution of $E_i$. 

To match notation with \ref{Kellerplus}, set $Q^\bullet_i:=LP_i^\bullet$, and write $P=\bigoplus P_i^\bullet$, $Q=\bigoplus Q_i^\bullet$,  and $I=\bigoplus I^i_\bullet$ with summations over $i=1,\hdots,n$.  Similarly write $T$ and $E$ for brevity. We claim that $\End^{\DG}_U(I)$ and $\End^{\DG}_U(LP)$ are quasi-isomorphic.  Let $LP[1]\to I$ be given by splicing the two resolutions $LP$ and $I$ of $LT$, using that $LT \cong E$. Now by adjunction
\begin{eqnarray*}
\Hom^{\DG}_U\!\big(LP[1],LP[1]\to I\big)\cong\Hom^{\DG}_\Lambda\!\big(P[1],F(LP[1]\to I)\big),\label{to show exact}
\end{eqnarray*}
so provided that this complex is exact,  the claimed quasi-isomorphism follows by \ref{Kellerplus}. Since $P$ is $h$-projective, it suffices to show that $F(LP[1]\to I)$ is exact. For this, consider the following diagram, whose top row shows the splicing which gives $F(LP[1]\to I)$.
\[
\begin{tikzpicture}[xscale=1]
\node (aL) at (-1.5,0) {};
\node (a0) at (0,0) {$ F L P^1$};
\node (a1) at (2,0) {$ F L P^0$};
\node (a2) at (4,-0.5) {$ F L T $};
\node (a3) at (6,0) {$ F I_0 $};
\node (a4) at (8,0) {$ F I_1 $};
\node (aR) at (9.5,0) {};
\node (bL) at (-1.5,-1) {};
\node (b0) at (0,-1) {$ P^1 $};
\node (b1) at (2,-1) {$ P^0 $};
\node (b2) at (4,-1.6) {$ T $};
\draw[->] (aL)--(a0);
\draw[->] (a0)--(a1);
\draw[->] (a1)--(a2);
\draw[->] (a2)--(a3);
\draw[->] (a3)--(a4);
\draw[->] (a4)--(aR);
\draw[dashed,->] (a1)--(a3);
\draw[->] (bL)--(b0);
\draw[->] (b0)--(b1);
\draw[->] (b1)--(b2);
\draw[<-] (a0)--node[sloped,left,anchor=north]{$\scriptstyle\sim$}(b0);
\draw[<-] (a1)--node[sloped,left,anchor=north]{$\scriptstyle\sim$}(b1);
\draw[<-] (a2)--node[sloped,left,anchor=north]{$\scriptstyle\sim$}(b2);
\end{tikzpicture}
\]
Now the complex
\[
\hdots \to P^1 \to P^0\to T\to 0
\]
is exact, being a projective resolution of $T$, and the complex 
\[
0\to FLT \to FI_0\to FI_1\to\hdots
\]
is exact since $\RHom_U(\cV,LT)\cong T$.  Combining these, the top complex $F(LP[1]\to I)$ is thus exact, establishing the claim.

Now it is  clear that
\[
\End^{\DG}_\Lambda(P)\to \End^{\DG}_U(LP)=\End^{\DG}_U(Q)
\]
is a quasi-isomorphism, and so combining we see that $\End^{\DG}_\Lambda(P)$ is quasi-isomorphic to $\End^{\DG}_U(I)$.  Finally, choose an injective resolution $0\to T_i\to J^i_\bullet$ of each $T_i$, and set $J:=\bigoplus_{i=1}^n J^i_\bullet$.  It is well known (again using \ref{Kellerplus}), that $\End^{\DG}_\Lambda(J)$ is quasi-isomorphic to $\End^{\DG}_\Lambda(P)$, and thus by the above to $\End^{\DG}_U(I)$.\\ 
(3) Consider $F:=\mathbb{F}^{-1}\circ\widehat{(-)}\colon\Mod\Lambda\to\Mod\AB$, which has exact left adjoint $L=\rest\circ\,\mathbb{F}$.  Being the composition of exact functors, $F$ is also exact, so $\mathbf{R}^{t}F(S_i)=0$ for all $t>0$ and for all $i=1,\hdots,n$.  Since $\mathbb{F}$ is an equivalence, and since the adjunction $\rest \dashv \widehat{(-)}$ restricts to an equivalence between finite length $\Lambda$-modules supported at $\m$ and finite length $\widehat{\Lambda}$-modules, it follows that the counit $L \circ F \to \Id$ is an isomorphism on the objects~$S_i$.  Hence the statement follows from \ref{FLlemma}.\\
(4) By \ref{Ed Lemma}, the DG deformation functor $\Def_{\AB}$ is isomorphic to the naive deformation functor in  \ref{def:def functor}. Thus, since each $S_i$ is one-dimensional, using the standard argument (see e.g.\ \cite[3.1]{DW1}) it follows that
\begin{align*}
\Def_{\AB}(\Gamma)
&=\left. \left \{ \begin{array}{cl} \bullet & \mbox{A left $\AB$-module structure on $M=(S_i\otimes_{\mathbb{C}}\Gamma_{ij})$ such that}\\
&\mbox{$(S_i\otimes_{\mathbb{C}}\Gamma_{ij})$ becomes a $\AB$-$\Gamma$ bimodule.}\\
\bullet &\mbox{A collection of isomorphisms }\delta_i\colon M\otimes_\Gamma (\Gamma /\m)e_i\xrightarrow{\sim} S_i
\end{array} \right\} \middle/ \sim \right. \\
&=\left. \left \{ \begin{array}{cl} \bullet & \mbox{A left $\AB$-module structure on $\Gamma$ such that $\Gamma\in\Mod\AB\otimes\Gamma^{\op}$.}\\
\bullet &\mbox{A collection of isomorphisms }\delta_i\colon \Gamma\otimes_\Gamma (\Gamma /\m)e_i\xrightarrow{\sim} S_i
\end{array} \right\} \middle/ \sim \right. \\
&=\Hom_{\alg_n}(\CA,\Gamma).
\end{align*}
It remains to show that $\CA\in\proart_n$. Since $f\colon \mathfrak{U}\to\Spec \mathfrak{R}$  is projective, it follows that $\AB$ is a module-finite $\mathfrak{R}$-algebra. Since $\mathfrak{R}$ is $\mathfrak{m}$-adically complete, so too is $\AB$ \cite[Thm 8.7]{Matsumura}. It is well-known, using Nakayama, that this implies that $\AB$ is also $J$-adically complete, where $J$ is the augmentation ideal of $\AB$. It follows that $\CA$ is also complete with respect to its augmentation ideal.
\end{proof}

\begin{remark}
The above proof of \ref{All defs prorep}\eqref{All defs prorep 2}  establishes that we can compute $\Def_U$ using the DGA associated to the specific locally free resolutions $\cV\otimes_\Lambda P_i^\bullet$ of the $E_i$.  It is rare to be able to compute the deformations of a sheaf using the DGA of a locally free resolution, and essentially it is this extra control of the deformation theory that allows us to prove the contraction theorem in \ref{contraction theorem} below.  Note that even taking the DGA of a different locally free resolution of the $E_i$ may give a different deformation functor.  
\end{remark}

\section{The Contraction Theorem}\label{section5}
In the global setup of \ref{setupglobal}, by \ref{All defs prorep} it follows that simultaneous noncommutative deformations of the reduced fibre is prorepresented by $\CA$, namely
\[
\Def_X\cong\Hom_{\proart_n}(\CA,-). 
\]
Further, by \ref{ME lemma}, $\CA$ is morita equivalent to $\widehat{\Lambda}_{\con}$.   We are thus motivated to control $\CA$ (and $\Lambda_{\con}$), and we now do this in a sequence of reduction steps, culminating in the contraction theorem of \S\ref{subsect contraction}.

\subsection{Behaviour of $\Lambda_{\con}$ under Global Sections and Base Change}
We revert to the Zariski local setup of \ref{setupZariski}. The following lemma is important: it is very well-known if $f$ is an isomorphism in codimension one \cite[4.2.1]{VdB1d}, however in our more general setting care is required.
\begin{lemma}\label{up=down VdB bundle}
With the setup as in \ref{setupZariski},
$\End_{U}(\cV^*)\cong \End_R(f_*(\cV^*))$ so 
\[
\Lambda\cong\End_R(R\oplus f_*(\cN^*))^{\op}\quad\mbox{and}\quad \Lambda_{\con}\cong\left(\End_R(R\oplus f_*(\cN^*))/[R]\right)^{\op}.
\] 
\end{lemma}
This allows us to reduce many problems to the base $\Spec R$. 
Indeed $\End_{U}(\cV)\ncong \End_R(f_*\cV)$ in general, so the statement and proof of \ref{up=down VdB bundle} is a little subtle.  Since $\cV^*$ is generated by global sections and is tilting, the proof of \ref{up=down VdB bundle} follows immediately from \ref{prop:up=down} below.  This requires the following easy well-known lemma.

\begin{lemma}\label{prelimforsheafup=down}
Assume that $f\colon Y\to \Spec S$ is a proper birational map between integral schemes, where $S$ is normal.
\begin{enumerate}
\item\label{prelimforsheafup=down 1} Suppose that $\cF\in\coh Y$ is generated by global sections.  Then we can find a morphism $\cO_Y^{\oplus a}\twoheadrightarrow \cF$ such that $S^{\oplus a}\twoheadrightarrow f_*\cF$.
\item\label{prelimforsheafup=down 2} Suppose that $\cW_1,\cW_2$ are finite rank vector bundles on $Y$. Then \[\Hom_Y(\cW_1,\cW_2)\hookrightarrow\Hom_S(f_*\cW_1,f_*\cW_2).\]
\end{enumerate}
\end{lemma}
\begin{proof}
(1) This is a basic consequence of Zariski's main theorem, together with the fact that $f_*$ preserves coherence.\\
(2) Since $\Hom_Y(\cW_1,\cW_2)=H^0(\cW_1^*\otimes\cW_2)$ and $\cW_1^*\otimes\cW_2$ is a vector bundle, by integrality $\Hom_Y(\cW_1,\cW_2)$ is a torsion-free $S$-module. The statement follows.
\end{proof}

\begin{prop}\label{prop:up=down}
Assume that $f\colon Y\to \Spec S$ is a proper birational map between integral schemes, where $S$ is normal.  If $\cW$ is a vector bundle on $Y$ of finite rank, which is generated by global sections, such that $\Ext_Y^1(\cW,\cW)=0$, then $\End_Y(\cW)\cong\End_S(f_*\cW)$.
\end{prop}
\begin{proof}
By \ref{prelimforsheafup=down}\eqref{prelimforsheafup=down 1} there is a short exact sequence
\begin{eqnarray}
0\to\cK\to\cO_Y^{\oplus a}\to\cW\to 0\label{up=down1}
\end{eqnarray}
such that 
\begin{eqnarray}
0\to f_*\cK\to S^{\oplus a}\to f_*\cW\to 0\label{up=down2}
\end{eqnarray}
is exact.  Applying $\Hom_Y(-,\cW)$ to (\ref{up=down1}) and applying $\Hom_S(-,f_*\cW)$ to (\ref{up=down2}), we have an exact commutative diagram
\[
\begin{tikzpicture}[scale=1,node distance=1]
\node (A0) at (-2.25,0) {$0$};
\node (A1) at (0,0) {$\Hom_{Y}(\cW,\cW)$};
\node (A2) at (3.5,0) {$\Hom_{Y}(\cO_Y^{\oplus a},\cW)$};
\node (A3) at (7,0) {$\Hom_{Y}(\cK,\cW)$};
\node (A4) at (10,0) {$\Ext^1_{Y}(\cW,\cW)=0$};
\node (B0) at (-2.25,-1.5) {$0$};
\node (B1) at (0,-1.5) {$\Hom_{S}(f_*\cW,f_*\cW)$};
\node (B2) at (3.5,-1.5) {$\Hom_{S}(S^{\oplus a},f_*\cW)$};
\node (B3) at (7,-1.5) {$\Hom_{S}(f_*\cK,f_*\cW)$};
\draw[->] (A0) -- (A1);
\draw[->] (A1) -- (A2);
\draw[->] (A2) -- (A3);
\draw[->] (A3) -- (A4);
\draw[->] (B0) -- (B1);
\draw[->] (B1) -- (B2);
\draw[->] (B2) -- (B3);
\draw[->] (A1) -- node[right] {$\scriptstyle \alpha$} (B1);
\draw[-,transform canvas={xshift=+0.15ex}] (B2) -- (A2);
\draw[-,transform canvas={xshift=-0.15ex}] (B2) -- (A2);
\draw[->] (A3) -- node[right] {$\scriptstyle \beta$} (B3);
\end{tikzpicture}
\]
where both $\alpha$ and $\beta$ are injective by \ref{prelimforsheafup=down}\eqref{prelimforsheafup=down 2}. By the snake lemma, $\alpha$ is also surjective.
\end{proof}
With the setup in \ref{setupZariski}, both $\Lambda$ and $\Lambda_{\con}$ have the structure of an $R$-module.  For our purposes later, we need to control this under flat base change.  In what follows, for $\p\in\Spec R$ we consider the base change squares
\[
\begin{array}{c}
\begin{tikzpicture}[yscale=1.25]
\node (Xpc) at (-3,0) {$\mathfrak{U}_\p$}; 
\node (Xp) at (-1,0) {$U_\p$}; 
\node (X) at (1,0) {$U$};
\node (Rpc) at (-3,-1) {$\Spec \mathfrak{R}_\p$}; 
\node (Rp) at (-1,-1) {$\Spec R_\p$}; 
\node (R) at (1,-1) {$\Spec R$};
\draw[->] (Xpc) to node[above] {$\scriptstyle m$} (Xp);
\draw[->] (Xp) to node[above] {$\scriptstyle k$} (X);
\draw[->] (Rpc) to node[above] {$\scriptstyle l$} (Rp);
\draw[->] (Rp) to node[above] {$\scriptstyle j$} (R);
\draw[->] (Xpc) --  node[left] {$\scriptstyle \theta$}  (Rpc);
\draw[->] (Xp) --  node[left] {$\scriptstyle \varphi$}  (Rp);
\draw[->] (X) --  node[right] {$\scriptstyle f$}  (R);
\end{tikzpicture}
\end{array}
\]
where $\mathfrak{R}_\p$ denotes the completion of $R_\p$ at its unique maximal ideal.

\begin{prop}\label{Lambda and Lambda con under FBC}
With the setup as in \ref{setupZariski},
\begin{enumerate}
\item\label{Lambda and Lambda con under FBC 1} $U_\p$ is derived equivalent to $\Lambda_\p$ via the tilting bundle $k^*\cV=\cO_{U_\p}\oplus k^*\cN$.
\item\label{Lambda and Lambda con under FBC 2} $\mathfrak{U}_\p$ is derived equivalent to $\widehat{\Lambda_\p}$ via the tilting bundle $m^*k^*\cV=\cO_{\mathfrak{U}_\p}\oplus m^*k^*\cN$.
\item\label{Lambda and Lambda con under FBC 3} $(\Lambda_{\con})_\p\cong (\Lambda_\p)_{\con}$. \phantom{$\widehat{\Lambda}$} 
\item\label{Lambda and Lambda con under FBC 4} $(\Lambda_{\con})_\p\otimes_{R_\p}\mathfrak{R}_\p\cong (\widehat{\Lambda_\p})_{\con}$.
\end{enumerate}
\end{prop}
\begin{proof}
All statements are elementary, but we give the proof for completeness.\\
(1) $U$ is derived equivalent to $\Lambda$ via the tilting bundle $\cV:=\cO_U\oplus\cN$.  It is well-known that this implies $U_\p$ is derived equivalent to $\Lambda_\p$ via the tilting bundle $k^*\cV=\cO_{U_\p}\oplus k^*\cN$.  For example, a proof of the Ext vanishing together with the fact that $\End_{U_\p}(k^*\cV)\cong \Lambda_\p$ can be found in \cite[4.3(2)]{IW5}.  For generation, first observe that $j$ is an affine morphism, hence so is $k$.  Then $\RHom_{U_\p}(k^*\cV,x)=0$ implies, by adjunction, that $\RHom_{U}(\cV,k_*x)=0$.  Since $\cV$ generates, $k_*x=0$ and so since $k$ is affine $x=0$.\\
(2) The proof is identical to (1).\\
(3) Since $f$ is proper, the category $\coh U$ is $R$-linear.  In particular $\Hom_U(\cO_U,\cV)$ is a finitely generated $\End_U(\cO_U)$-module, so we can find a surjection
\begin{eqnarray}
\Hom_U(\cO_U,\cO_U)^{\oplus a}\to \Hom_U(\cO_U,\cV)\to 0.\label{geom approx O 1}
\end{eqnarray}
Tracking the images of the identities on the left hand side under the above map gives elements $g_1\hdots,g_a\in\Hom_U(\cO_U,\cV)$, and so we may use these to form a natural morphism 
\begin{eqnarray}
\cO_U^{\oplus a}\xrightarrow{h} \cV\label{geom approx O 2}
\end{eqnarray}
such that applying $\Hom_U(\cO_U,-)$ to \eqref{geom approx O 2} gives \eqref{geom approx O 1}.  By definition, this means that $h$ is an $\add\cO_U$-approximation of $\cV$.  Hence applying $\Hom_U(\cV,-)$ to \eqref{geom approx O 2} yields an exact sequence
\begin{eqnarray}
\Hom_U(\cV,\cO_U)^{\oplus a}\to \Hom_U(\cV,\cV)\to \Lambda_\con\to 0.\label{geom approx O 3}
\end{eqnarray}
Interpreting $\Hom_U(-,-)=f_*\sHom(-,-)$, applying the exact functor $j^*(-)=(-)_\p$ to \eqref{geom approx O 3} and using flat base change gives the exact sequence
\[
\varphi_*k^*\sHom_{U}(\cV,\cO_U)^{\oplus a}\to \varphi_*k^*\sHom_U(\cV,\cV)\to (\Lambda_\con)_\p\to 0.
\]
Since $\cV$ is coherent we may move the $k^*$ inside $\sHom$, and so the above is simply
\begin{eqnarray}
\Hom_{U_\p}(k^*\cV,\cO_{U_\p})^{\oplus a}\to \Hom_{U_\p}(k^*\cV,k^*\cV)\to (\Lambda_\con)_\p\to 0.\label{geom approx O 5}
\end{eqnarray}
But on the other hand applying the exact functor $k^*$ to \eqref{geom approx O 2} gives a morphism 
\begin{eqnarray}
\cO_{U_\p}^{\oplus a}\xrightarrow{k^*(h)} k^*\cV\label{geom approx O 6}
\end{eqnarray}
and further applying $j^*$ to \eqref{geom approx O 1} and using flat base change shows that 
\[
\Hom_{U_\p}(\cO_{U_\p},\cO_{U_\p})^{\oplus a}\to \Hom_{U_\p}(\cO_{U_\p},k^*\cV)\to 0
\]
is exact.  Hence $k^*(h)$ is an $\add\cO_{U_\p}$-approximation, and so applying $\Hom_{U_\p}
(k^*\cV,-)$ to \eqref{geom approx O 6} gives an exact sequence 
\begin{eqnarray}
\Hom_{U_\p}(k^*\cV,\cO_{U_\p})^{\oplus a}\to \Hom_{U_\p}(k^*\cV,k^*\cV)\to (\Lambda_\p)_\con\to 0.\label{geom approx O 8}
\end{eqnarray}
Combining \eqref{geom approx O 5} and \eqref{geom approx O 8} shows that $(\Lambda_\p)_\con\cong (\Lambda_\con)_\p$, as required.\\
(4) The proof is identical to (3).
\end{proof}

\subsection{The Contraction Theorem}\label{subsect contraction}
In this subsection, so as to be able to work globally in future papers, we first relate the support of $\Lambda_{\con}$ to the locus $L$. We then control the support of $\CA$ to deduce the deformation theory corollaries.

We need the following fact, which is well-known.

\begin{lemma}\label{notequiv lemma}
Suppose that $f\colon Y\to Z$ is a morphism of noetherian schemes which is not an isomorphism.  Then $\RDerived f_*\colon \D(\Qcoh Y)\to\D(\Qcoh Z)$ is not an equivalence.
\end{lemma}
\begin{proof}
If $\RDerived f_*$ is an equivalence then its inverse is necessarily given by its adjoint $\LDerived f^*$.  Since noetherian schemes are quasi-compact and quasi-separated, both unbounded derived categories are compactly generated triangulated categories (with compact objects the perfect complexes),  and so the above equivalence restricts to an equivalence
\[
\Perf(Y)\xleftarrow{\sim}\Perf(Z)\colon \LDerived f^*.
\]
By Balmer \cite[9.7]{Balmer} it follows that $f$ is an isomorphism, which is a contradiction.
\end{proof}
Leading up to the next lemma, choose a closed point $\m\in L$, and pick an affine open $\Spec R$ containing $\m$.  For any $\p\in\Spec R$ we base change to obtain the following diagram.
\[
\begin{array}{c}
\begin{tikzpicture}[yscale=1.25]
\node (Up) at (-1,0) {$U_\p$}; 
\node (U) at (1,0) {$U$};
\node (Rp) at (-1,-1) {$\Spec R_\p$}; 
\node (R) at (1,-1) {$\Spec R$};
\draw[->] (Up) to node[above] {$\scriptstyle k$} (U);
\draw[->] (Rp) to node[above] {$\scriptstyle j$} (R);
\draw[->] (Up) --  node[left] {$\scriptstyle \varphi_\p$}  (Rp);
\draw[->] (U) --  node[right] {$\scriptstyle f$}  (R);
\end{tikzpicture}
\end{array}
\]
Note that by flat base change the morphism $\varphi_\p$ is still projective birational, and satisfies $\RDerived{\varphi_\p}_*\cO_{U_\p}=\cO_{R_\p}$.
\begin{lemma}\label{key supp lemma}
With the global setup of \ref{setupglobal}, and notation as above,
\[
\varphi_\p\mbox{ is not an isomorphism}\iff \p\in\Supp_R\Lambda_{\con}.
\]
\end{lemma}
\begin{proof}
As in \cite[4.6]{KIWY}, it is easy to see that the diagram 
\begin{equation}
\begin{array}{c}
\begin{tikzpicture}
\node (roof) at (0,0) {$\D(\Qcoh U_\p)$};
\node (U) at (7,0) {$\D(\Mod\Lambda_\p)$};
\node (U') at (0,-1.5) {$\D(\Qcoh \Spec R_\p)$};
\node (base) at (7,-1.5) {$\D(\Mod R_\p)$};
\draw[->] (roof) -- node[above] {$\scriptstyle \Psi:=\RHom_{U_\p}(\cO_{U_\p}\oplus k^*\cN,-)$} node[below] {$\scriptstyle \sim$} (U);
\draw[->] (roof) --  node[left] {$\scriptstyle \RDerived{\varphi_\p}_*$}  (U');
\draw[->] (U) --  node[right] {$\scriptstyle e(-)$} (base);
\draw[-,transform canvas={yshift=+0.15ex}] (U') -- (base);
\draw[-,transform canvas={yshift=-0.15ex}] (U') -- (base);
\end{tikzpicture}
\end{array}\label{comm flat Db}
\end{equation}
commutes, where the top functor is an equivalence by \ref{Lambda and Lambda con under FBC}, and by abuse of notation $e$ also denotes the idempotent in $\Lambda_\p$ corresponding to $\cO_{U_\p}$.\\
($\Rightarrow$) To ease notation, we drop $\p$ and write $\varphi$ for $\varphi_\p$. Since $\RDerived\varphi_*$ is not an equivalence by \ref{notequiv lemma}, we may find some $x\in\D(\Qcoh U_\p)$ such that the counit
\[
\LDerived\varphi^*\RDerived\varphi_*(x)\xrightarrow{\varepsilon_x} x
\]
is not an isomorphism, and so the object $c := \operatorname{Cone}(\varepsilon_x)$ is non-zero. Now we argue that $\RDerived\varphi_*(c)=0$, which is equivalent to the morphism $\RDerived\varphi_*(\varepsilon_x)$ being an isomorphism. But
\[
\RDerived\varphi_*(x)\xrightarrow{\eta_{\RDerived\varphi_*(x)}} \RDerived\varphi_*\LDerived\varphi^*\RDerived\varphi_*(x)\xrightarrow{\RDerived\varphi_*(\varepsilon_x)} \RDerived\varphi_*(x)
\]
gives the identity map by a triangular identity, and $\eta_{\RDerived\varphi_*(x)}$ is an isomorphism since it is well known that $\eta\colon \Id\to\RDerived\varphi_*\LDerived\varphi^*$ is a functorial isomorphism by the projection formula.   Thus indeed $c$ is a non-zero object such that $\RDerived\varphi_*(c)=0$.

Now across the top equivalence in \eqref{comm flat Db}, since $c\neq0$ it follows that $\Psi(c)\neq0$ and so $H^j(\Psi(c))\neq 0$ for some $j$.  Further since the diagram \eqref{comm flat Db} commutes, $e\Psi(c)=0$, so since $e(-)$ is exact we deduce that $eH^i(\Psi(c))=0$ for all $i$. In particular there exists a non-zero $\Lambda_\p$-module $M:=H^j(\Psi(c))$ such that $eM=0$.  It follows that $M$ is a non-zero module for $(\Lambda_\p)_\con$, hence necessarily $(\Lambda_\p)_{\con}\neq 0$.   But by \ref{Lambda and Lambda con under FBC} we have $(\Lambda_\p)_{\con}\cong (\Lambda_{\con})_{\p}$, hence $\p\in\Supp_R\Lambda_{\con}$.  \\
($\Leftarrow$) If $\p\in\Supp_R\Lambda_{\con}$, by \ref{Lambda and Lambda con under FBC} $(\Lambda_\p)_{\con}\neq 0$.  Hence the right hand functor $e(-)$ in \eqref{comm flat Db} is not an equivalence, and so the left hand functor cannot be an equivalence either.  By \ref{notequiv lemma}, it follows that $\varphi_\p$ is not an isomorphism.
\end{proof}

\begin{thm}\label{contract1}
With the global setup of \ref{setupglobal}, choose a closed point $\m\in L$, and pick an affine open $\Spec R$ containing $\m$.  Then
\begin{enumerate}
\item\label{contract1 1} $\Supp_R\Lambda_{\con}=L_R := L \cap \Spec R$.
\item\label{contract1 2}  $\Supp_{\mathfrak{R}}\CA=\{ \p\in L_R\mid \p\subseteq \m\}$.
\end{enumerate}
\end{thm}
\begin{proof}
(1) It is clear that $L_R=\{\p\in\Spec R\mid \varphi_\p\mbox{ is not an isomorphism}\}$, and so the result is immediate from \ref{key supp lemma}. \\
(2) Applying the above argument to the morphism $\mathfrak{U}\to\Spec \mathfrak{R}$, in an identical manner $\Supp_{\mathfrak{R}}\CA=L_{\mathfrak{R}} := L \cap \Spec \mathfrak{R}$.  It is clear that $L_{\mathfrak{R}}=\{ \p\in L_R\mid \p\subseteq \m\}$.
\end{proof}

The following is an immediate corollary.

\begin{cor}[Contraction Theorem]\label{contraction theorem}
In the Zariski local setup $f\colon U\to\Spec R$ of \ref{setupZariski}, suppose further that $\dim U=3$. Then
\[
\mbox{$f$ contracts curves without contracting a divisor}\iff \dim_{\mathbb{C}}\Lambda_\con<\infty.
\]
\end{cor}
\begin{proof}
In this setting $f$ contracts curves without contracting a divisor if and only if $L_R$ is a zero-dimensional scheme.  The result follows from \ref{contract1}\eqref{contract1 1}.
\end{proof}

\begin{remark} 
It follows from \ref{contraction theorem} (or indeed the morita equivalence in \ref{ME lemma}) that the condition $\dim_{\mathbb{C}}\Lambda_\con<\infty$ is independent of the choice of $\Lambda$, and thus the choice of tilting bundle of the form $\cO\oplus\cN$.  Hence we may make any choice, and detect the contractibility by calculating the resulting $\dim_{\mathbb{C}}\Lambda_\con$.  However, to get well-defined invariants that do not depend on choices, we pass to the formal fibre $\mathfrak{R}$ and use the algebra $\CA$. 
\end{remark}

\begin{cor}\label{contraction theorem def functor}
In the global setup $f\colon X\to X_{\con}$ of \ref{setupglobal}, suppose further that $\dim X=3$, pick a closed point $\m\in L$ and set $C:=f^{-1}(\m)$. The following are equivalent.
\begin{enumerate}
\item There is a neighbourhood of $\m$ over which $f$ does not contract a divisor.
\item The functor $\Def_X$ of simultaneous noncommutative deformations of the reduced fibre $C^{\redu}$ is representable.
\end{enumerate}
\end{cor}
\begin{proof}
By choosing an affine open $\Spec R$ containing $\m$, this is identical to the proof of~\ref{contraction theorem}, appealing to \ref{contract1}\eqref{contract1 2} instead of \ref{contract1}\eqref{contract1 1}, and using the prorepresentability of $\Def_X$ from~\ref{All defs prorep}.
\end{proof}

\section{Deformations of the Scheme-Theoretic Fibre}\label{section:schemefibre}

In the global setup $f\colon X\to X_{\con}$ of \ref{setupglobal}, we choose a closed point $\m\in L$ and in this section study commutative and noncommutative deformations of the scheme-theoretic fibre $\cO_C$, where $C:=f^{-1}(\m)$.  We show that commutative and noncommutative deformations are prorepresented by the same object, and more remarkably that the prorepresenting object can be obtained from the same $\AB$ as $\CA$ can.  This allows us to relate deformations of the reduced and scheme-theoretic fibres, in a way that otherwise would not be possible.

\subsection{Prorepresentability}
With the Zariski local setup in \ref{setupZariski}, taking the dual bundle in~\eqref{derived equivalence} induces a derived equivalence
\begin{eqnarray}
\begin{array}{c}
\begin{tikzpicture}
\node (a1) at (0,0) {\phantom{.}$\Db(\coh U)$};
\node (a2) at (5,0) {$\Db(\mod \End_U(\cV^*))$.};
\draw[->] (a1) -- node[above] {$\scriptstyle\RHom_U(\cV^*,-)$} node [below] {$\scriptstyle\sim$} (a2);
\end{tikzpicture}
\end{array}\label{derived equivalence dual}
\end{eqnarray}
Note that $\End_U(\cV^*) = \Lambda^{\op}$.  Also, by \cite[3.5.7]{VdB1d}, under the above equivalence \eqref{derived equivalence dual} the sheaf $\cO_C$ corresponds to a simple $\Lambda^{\op}$-module, which we denote $T_0^\prime$.  This fact is the reason we pass to $\cV^*$, since it will allow us to easily apply \ref{Kellerplus} in the proof of \ref{oCdefmain} below.

Passing to the formal fibre $\mathfrak{U}\to\Spec\mathfrak{R}$, the dual of the previous bundle in \ref{def basic algebra} gives the following natural definition. 

\begin{defin}\label{def basic algebra dual}
We write $\cM:=\bigoplus_{i=1}^n \cM_i$ and define 
\[
\BB:=\End_{\mathfrak{U}}(\cO_{\mathfrak{U}}\oplus \cM)=\AB^{\op},
\]
which is the basic algebra morita equivalent to $\widehat{\Lambda}^{\op}$. From this, we define 
\[
\BB_{\fib}:=\End_{\mathfrak{U}}(\cO_{\mathfrak{U}}\oplus \cM)/[\cM].
\]
\end{defin}
Under this dual setup, the following is obvious.
\begin{lemma}\label{BfibAfib}
$\BB^{\phantom \op}_{\fib}\cong\AB_{\fib}^{\op}\cong\AB^{\phantom \op}_{\fib}$, and in particular $\BB_{\fib}$ is commutative.
\end{lemma}
\begin{proof}
The first statement is clear since $\BB=\AB^{\op}$.  The second statement is \ref{Afibcomm}.
\end{proof}

In what follows, we let $S_0^\prime$ denote the simple $\BB$-module corresponding to $T_0^\prime$ under the composition of the completion functor and the morita equivalence between $\widehat{\Lambda}^{\op}$ and $\BB$. Similarly to \ref{TandSnotation}, for a given scheme-theoretic fibre $C$, below we use the following notation.
\begin{enumerate}
\item $\Def_X^{\cO_C}$ for the DG deformation functor associated to the injective resolution of~$\cO_{C}\in\coh X$.
\item $\Def_U^{\cO_C}$ for the DG deformation functor associated to the injective resolution of~$\cO_{C}\in\coh U$.
\item $\Def_{\Lambda^{\op}}^{T_0^\prime}$ for the DG deformation functor associated to the injective resolution of~$T_0^\prime\in\mod \Lambda^{\op}$.
\item $\Def_{\BB}^{S_0^\prime}$ for the DG deformation functor associated to the injective resolution of~$S_0\in\mod\BB$.
\end{enumerate}

\begin{thm}\label{oCdefmain}
In the global setup $f\colon X\to X_{\con}$ of \ref{setupglobal}, pick a closed point $\m\in L$ and set $C:=f^{-1}(\m)$. Then
\[
\Def^{\cO_C}_X\cong\Hom_{\proart_1}(\AB_{\fib},-),
\] 
and further $\AB_{\fib}$ is commutative.
\end{thm}
\begin{proof}
Exactly as in \ref{All defs prorep}, we first claim that
\begin{eqnarray}
\Def_X^{\cO_C}\cong\Def_U^{\cO_C}\cong\Def_{\Lambda^{\op}}^{T_0^\prime}\cong \Def_\BB^{S_0^\prime}\cong \Hom_{\proart_1}(\BB_{\fib},-).\label{claimed isos}
\end{eqnarray}
Under $i\colon U\hookrightarrow X$, since $C$ is closed $\RDerived i_*\cO_C=i_*\cO_C$, and so the first claimed isomorphism follows from \ref{FLlemma}.  Under the derived equivalence \eqref{derived equivalence dual} above, the sheaf $\cO_C$ corresponds to the module $T_0^\prime$, so the second claimed isomorphism follows from \ref{Kellerplus}, exactly as in the proof of \ref{All defs prorep}\eqref{All defs prorep 2}.  The third claimed isomorphism follows from the fact that  finite length $\Lambda^{\op}$-modules supported at $\m$ are equivalent to finite length $\BB$-modules, and so the result follows from \ref{FLlemma}.  For the last claimed isomorphism, since $S_0^\prime$ is the vertex simple, as in \cite[3.1]{DW1} it is clear that
\[
\Def_\BB^{S_0^\prime}\cong \Hom_{\alg_1}(\BB_{\fib},-).
\] 
It remains to show that $\BB_{\fib}\in\proart_1$, but this holds since $\mathfrak{R}$, thus $\BB$, and thus $\BB_{\fib}$, are complete with respect to their augmentation ideals. With \eqref{claimed isos} established, the remaining statement follows from \ref{BfibAfib}.
\end{proof}

\subsection{Comparison of Deformations}\label{subsect:comparison}
One of the remarkable consequences of \ref{All defs prorep} and \ref{oCdefmain} is that there is a single $\AB$, of which various factors control different natural geometric deformation functors.  Thus proving elementary facts for the ring $\AB$ has strong deformation theory consequences; the following is one such example.
\begin{prop}\label{dimgivesdim}
If $\dim_{\mathbb{C}}\CA<\infty$, then 
$\dim_\mathbb{C}\AB_{\fib}<\infty$.
\end{prop}
\begin{proof}
Since $\cO_{\mathfrak{U}}\oplus\cM$ is generated by global sections, by \ref{prop:up=down} $\AB\cong \End_{\mathfrak{R}}(\mathfrak{R}\oplus \clocCon_*\cM)^{\op}$.  To~ease notation we temporarily set $D=\clocCon{}_*\cM$, so $\CA^{\op}=\End_{\mathfrak{R}}(\mathfrak{R}\oplus D)/[\mathfrak{R}]$ and $\AB_{\fib}^{\op}=\End_{\mathfrak{R}}(\mathfrak{R}\oplus D)/[D]$.  Since taking opposite rings does not affect dimension, in what follows we can ignore the ops.

To establish the result, we prove the contrapositive.  If $\dim_{\mathbb{C}}\AB_{\fib}=\infty$ then there exists a non-maximal prime ideal $\p\in\Spec\mathfrak{R}$ such that $(\AB_{\fib})_\p\neq 0$. As in \ref{Lambda and Lambda con under FBC} we have $(\AB_{\fib})_\p\cong {(\AB_\p)}_{\fib}$, and thus ${(\AB_\p)}_{\fib} = \End_{\mathfrak{R}_\p}(\mathfrak{R}_\p\oplus D_\p)/[D_\p]\neq 0$.  Certainly this means that $D_\p$ cannot be free.  Now if $\Id\colon D_\p\to D_\p$ factors through $\add\mathfrak{R}_\p$ then $D_\p$ is projective.  But since $\mathfrak{R}_\p$ is local, $D_\p$ would then be free, which is a contradiction.  Thus $\Id\colon D_\p\to D_\p$ does not factor through $\add\mathfrak{R}_\p$, so 
\[
(\CA)_\p\cong {(\AB_\p)}_{\con} = \End_{\mathfrak{R}_\p}(\mathfrak{R}_\p\oplus D_\p)/[\mathfrak{R}_\p]\neq 0.
\]
This implies that $\p\in\Supp_{\mathfrak{R}}\CA$ and so $\dim_{\mathbb{C}}\CA=\infty$.
\end{proof}

\begin{cor}\label{obvious}
In the global setup $f\colon X\to X_{\con}$ of \ref{setupglobal}, pick a closed point $\m\in L$ and set $C:=f^{-1}(\m)$.  Write $C^{\redu}=\bigcup_{i=1}^nC_i$, then 
\begin{enumerate}
\item\label{obvious 1}If $\Def_X$ is representable, so is $\Def_X^{\cO_C}$.
\item\label{obvious 2} Suppose $\dim X=3$.  If $\Def_X^{\cO_C}$ is not representable, then $f$ contracts a divisor over a neighbourhood of $\m$.
\end{enumerate}
\end{cor}
\begin{proof}
(1) This is immediate from \ref{dimgivesdim}, since $\AB_{\fib}$ prorepresents $\Def_X^{\cO_C}$ by \ref{oCdefmain}, and $\CA$ prorepresents $\Def_X$ by \ref{All defs prorep}.\\
(2) Follows from \eqref{obvious 1} and \ref{contraction theorem def functor}.
\end{proof}
The converse to \ref{obvious}\eqref{obvious 2} is however false; we show this in \S\ref{examples section} below.   The following summarises the main results in this paper in the case of $3$-folds.

\begin{summary}\label{representability and algebra algebra properties}
In the global setup $f\colon X\to X_{\con}$ of \ref{setupglobal}, suppose further that $\dim X=3$, pick a closed point $\m\in L$, set $C:=f^{-1}(\m)$ and write $C^{\redu}=\bigcup_{i=1}^nC_i$.
\begin{enumerate}
\item\label{representability and algebra algebra properties 1} Both the noncommutative deformation functor $\Def_X^{\cO_{C}}$ and commutative deformation functor $\cDef_X^{\cO_{C}}$ are prorepresented by $\CAR$.  
\item\label{representability and algebra algebra properties 2} The following are equivalent.
\begin{enumerate}
\item[(a)] The functors $\cDef_X^{\cO_C}$ and $\Def_X^{\cO_C}$ are representable.
\item[(b)] $\dim_{\mathbb{C}}\CAR<\infty$.
\end{enumerate}
\item\label{representability and algebra algebra properties 3} The following are equivalent.
\begin{enumerate}
\item[(a)] The functor $\Def_X$ of simultaneous noncommutative deformations of the reduced fibre $C^{\redu}$ is representable.
\item[(b)] $\dim_{\mathbb{C}}\CA<\infty$.
\item[(c)] There is a neighbourhood of $\m$ over which $f$ does not contract a divisor.
\end{enumerate}
\item\label{representability and algebra algebra properties 4} The statements in \t{(3)} imply the statements in \t{(2)}, but the statements in \t{(2)} do not imply the statements in \t{(3)} in general.
\end{enumerate}
\end{summary}
\begin{proof}
Part (1) is \ref{oCdefmain}, and part (2) is tautological.  Part (3) is \ref{All defs prorep} and \ref{contraction theorem def functor}, and part (4) is shown by the counterexample in \ref{trihedral1} below.  
\end{proof}

\section{Examples}\label{examples section}
In this section, we first illustrate some $\CA$ that can arise in the setting of \ref{contraction theorem def functor} for specific $cA_n$ singularities.  We then show that the converse to \ref{obvious} is false, and also that \ref{contraction theorem def functor} fails if we replace noncommutative deformations by commutative ones.

\subsection{First Examples}
Consider the $cA_n$ singularities $\mathfrak{R}:=\mathbb{C}[[u,v,x,y]]/(uv-f_1\hdots f_n)$ for some $f_i\in\m:=(x,y)\subset \mathbb{C}[[x,y]]$.  The algebra $\AB$ in \ref{def basic algebra}, and thus the algebras $\CA$ and $\CAR$, can be obtained using the calculation in \cite[5.29]{IW6}.  Here we make this explicit in two examples.

\begin{example}[A $2$-curve flop] Consider the case $f_1=x$, $f_2=y$ and $f_3=x+y$.  In this example there are six crepant resolutions of $\Spec\mathfrak{R}$, and each has two curves above the origin. For one such choice, by \cite[5.2]{IW5}
\[
\AB=\End_{\mathfrak{R}}(\mathfrak{R}\oplus (u,x)\oplus (u,xy)),
\]
which by \cite[5.29]{IW6} can be presented as the completion of the quiver with relations
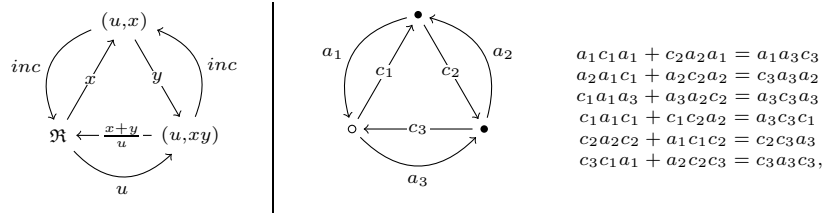
\[
\begin{array}{c|c}
\begin{array}{c}
\begin{tikzpicture} [bend angle=45, looseness=1.2]
\node[name=s,regular polygon, regular polygon sides=3, minimum size=2cm] at (0,0) {}; 
\node (C1) at (s.corner 1)  {$\scriptstyle (u,x)$};
\node (C2) at (s.corner 2)  {$\scriptstyle \mathfrak{R}$};
\node (C3) at (s.corner 3)  {$\scriptstyle (u,xy)$};
\draw[->] (C2) -- node  [gap] {$\scriptstyle x$} (C1); 
\draw[->] (C1) -- node  [gap] {$\scriptstyle y$} (C3);
\draw[->] (C3) -- node  [gap,pos=0.4] {$\scriptstyle \frac{x+y}{u}$} (C2); 
\draw [->,bend right] (C1) to node [left] {$\scriptstyle { inc}$} (C2);
\draw [->,bend right] (C2) to node [below] {$\scriptstyle u$} (C3);
\draw [->,bend right] (C3) to node [right] {$\scriptstyle { inc}$} (C1);
\end{tikzpicture}\end{array}
 &
\begin{array}{cc}
\begin{array}{c}
\begin{tikzpicture} [bend angle=45,looseness=1.2]
\node[name=s,regular polygon, regular polygon sides=3, minimum size=2cm] at (0,0) {}; 
\node (C1) at (s.corner 1)  [vertex] {};
\node (C2) at (s.corner 2)  [cvertex] {};
\node (C3) at (s.corner 3)  [vertex] {};
\draw[->] (C2) -- node  [gap] {$\scriptstyle c_1$} (C1); 
\draw[->] (C1) -- node  [gap] {$\scriptstyle c_2$} (C3);
\draw[->] (C3) -- node  [gap] {$\scriptstyle c_3$} (C2); 
\draw [->,bend right] (C1) to node [left] {$\scriptstyle a_1$} (C2);
\draw [->,bend right] (C2) to node [below] {$\scriptstyle a_3$} (C3);
\draw [->,bend right] (C3) to node [right] {$\scriptstyle a_2$} (C1);
\end{tikzpicture}\end{array}&
 {\scriptsize{
\begin{array}{c}
a_1c_1a_1+c_2a_2a_1=a_1a_3c_3\phantom{,}\\
a_2a_1c_1+a_2c_2a_2=c_3a_3a_2\phantom{,}\\
c_1a_1a_3+a_3a_2c_2=a_3c_3a_3\phantom{,}\\
c_1a_1c_1+c_1c_2a_2=a_3c_3c_1\phantom{,}\\
c_2a_2c_2+a_1c_1c_2=c_2c_3a_3\phantom{,}\\
c_3c_1a_1+a_2c_2c_3=c_3a_3c_3,
\end{array}}}
\end{array}
\end{array}
\]
given by the superpotential 
\[
W=\frac{1}{2}c_1a_1c_1a_1+\frac{1}{2}c_2a_2c_2a_2+\frac{1}{2}c_3a_3c_3a_3+c_1c_2a_2a_1-c_1a_1a_3c_3-c_2c_3a_3a_2.
\]
From this presentation, factoring by the appropriate idempotents  it is immediate that $\CAR\cong \mathbb{C}$, and further
\[
\CA\cong
\begin{array}{cc}
\begin{array}{c}
\begin{tikzpicture}[bend angle=20, looseness=1]
\node (a) at (-1.5,0) [vertex] {};
\node (b) at (0,0) [vertex] {};
\node (a1) at (-1.5,-0.2) {$\scriptstyle 1$};
\node (b1) at (0,-0.2) {$\scriptstyle 2$};
\draw[->,bend left] (b) to node[below] {$\scriptstyle a_2$} (a);
\draw[->,bend left] (a) to node[above] {$\scriptstyle c_2$} (b);
\end{tikzpicture}
\end{array}
& 
{\scriptsize
\begin{array}{c}
a_2c_2a_2=0\phantom{.}\\
c_2a_2c_2=0.
\end{array}}
\end{array}
\]
Since $\CA$ is finite dimensional, by \ref{representability and algebra algebra properties}\eqref{representability and algebra algebra properties 3} the contraction only contracts curves to a point.
\end{example}

\begin{example}[A divisorial contraction]\label{div example} Consider the case $f_1=f_2=x$ and $f_3=y$.  In this case there are three crepant resolutions, and each has two curves above the origin.  For one such choice, obtained by blowing up the ideal $u=v=x=0$, the resolution is sketched as follows,
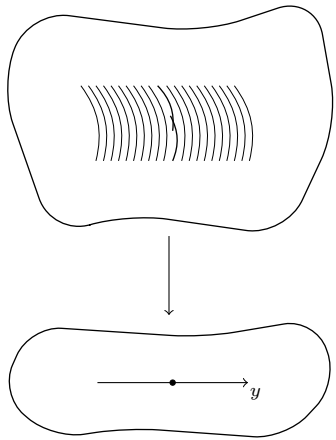
\[
\begin{tikzpicture}
\node at (0,0) {\begin{tikzpicture}[scale=1]
\coordinate (T) at (1.9,2);
\coordinate (TM) at (2.12-0.02,1.5-0.1);
\coordinate (BM) at (2.12-0.05,1.5+0.1);
\coordinate (B) at (2.1,1);
\draw[line width=0.5pt]\opt{colordiag}{[red]} (T) to [bend left=25] (TM);
\draw[line width=0.5pt]\opt{colordiag}{[red]} (BM) to [bend left=25] (B);

\foreach \y in {0.1,0.2,...,1}{ 
\draw[very thin]\opt{colordiag}{[blue]} ($(T)+(\y,0)+(0.02,0)$) to [bend left=25] ($(B)+(\y,0)+(0.02,0)$);
\draw[very thin]\opt{colordiag}{[blue]} ($(T)+(-\y,0)+(-0.02,0)$) to [bend left=25] ($(B)+(-\y,0)+(-0.02,0)$);;}
\draw[rounded corners=15pt,line width=0.5pt] (0.5,0) -- (1.5,0.3)-- (3.6,0) -- (4.3,1.5)-- (4,3.2) -- (2.5,2.7) -- (0.2,3) -- (-0.2,2)-- cycle;
\end{tikzpicture}};
\node at (0,-3.5) {\begin{tikzpicture}[scale=1]
\draw [->,black] (1.1,0.75) -- (3.1,0.75);
\filldraw\opt{colordiag}{[red]} (2.1,0.75) circle (1pt);
\node at (3.2,0.6) {$\scriptstyle y$};

\draw[rounded corners=12pt,line width=0.5pt] (0.5,0) -- (1.5,0.15)-- (3.6,0) -- (4.3,0.75)-- (4,1.6) -- (2.5,1.35) -- (0.2,1.5) -- (-0.2,0.6)-- cycle;
\end{tikzpicture}};
\draw[->] (0,-1.6) -- (0,-2.65);
\end{tikzpicture}
\]
where above the origin there are two curves, and every other fibre over the $y$-axis contains only one curve.  For this resolution
\[
\AB=\End_{\mathfrak{R}}(\mathfrak{R}\oplus (u,x)\oplus(u,xy))
\]
which by \cite[5.29]{IW6} can be presented as the completion of the quiver with relations
\[
\begin{array}{c|c}
\begin{array}{c}
\begin{tikzpicture} [bend angle=45, looseness=1.2]
\node[name=s,regular polygon, regular polygon sides=3, minimum size=2cm] at (0,0) {}; 
\node (C1) at (s.corner 1) {$\scriptstyle (u,x)$};
\node (C2) at (s.corner 2) {$\scriptstyle \mathfrak{R}$};
\node (C3) at (s.corner 3) {$\scriptstyle (u,xy)$};
\draw[->] (C2) -- node  [gap] {$\scriptstyle x$} (C1); 
\draw[->] (C1) -- node  [gap] {$\scriptstyle y$} (C3);
\draw[->] (C3) -- node  [gap,pos=0.4] {$\scriptstyle \frac{x}{u}$} (C2); 
\draw [->,bend right] (C1) to node [left] {$\scriptstyle inc$} (C2);
\draw [->,bend right] (C2) to node [below] {$\scriptstyle u$} (C3);
\draw [->,bend right] (C3) to node [right] {$\scriptstyle inc$} (C1);
\node (C2a) at ($(s.corner 2)+(-135:2pt)$) {};
\draw[<-]  (C2a) edge [in=-100,out=-170,loop,looseness=10] node[below] {$\scriptstyle y$} (C2a);
\end{tikzpicture}
\end{array}
 &
\begin{array}{cc}
\begin{array}{c}
\begin{tikzpicture} [bend angle=45,looseness=1.2]
\node[name=s,regular polygon, regular polygon sides=3, minimum size=2cm] at (0,0) {}; 
\node (C1) at (s.corner 1)  [vertex] {};
\node (C2) at (s.corner 2)  [cvertex] {};
\node (C3) at (s.corner 3)  [vertex] {};
\draw[->] (C2) -- node  [gap] {$\scriptstyle c_1$} (C1); 
\draw[->] (C1) -- node  [gap] {$\scriptstyle c_2$} (C3);
\draw[->] (C3) -- node  [gap] {$\scriptstyle c_3$} (C2); 
\draw [->,bend right] (C1) to node [left] {$\scriptstyle a_1$} (C2);
\draw [->,bend right] (C2) to node [below] {$\scriptstyle a_3$} (C3);
\draw [->,bend right] (C3) to node [right] {$\scriptstyle a_2$} (C1);
\node (C2a) at ($(s.corner 2)+(-135:2pt)$) {};
\draw[<-]  (C2a) edge [in=-100,out=-170,loop,looseness=10] node[below] {$\scriptstyle y$} (C2a);
\end{tikzpicture}\end{array}&
 {\scriptsize{
\begin{array}{l}
yc_1=c_1c_2a_2\\
a_1y=c_2a_2a_1\\
ya_3=a_3a_2c_2\\
c_3y=a_2c_2c_3\\
c_1a_1=a_3c_3\\
a_2a_1c_1=c_3a_3a_2\\
c_2c_3a_3=a_1c_1c_2
\end{array}}}
\end{array}
\end{array}
\]
From this, we see immediately that $\CAR\cong\mathbb{C}[[y]]$, and $\CA$ is the completion of the quiver
\[
\begin{array}{cc}
\begin{array}{c}
\begin{tikzpicture}[bend angle=20, looseness=1]
\node (a) at (-1.5,0) [vertex] {};
\node (b) at (0,0) [vertex] {};
\node (a1) at (-1.5,-0.2) {$\scriptstyle 1$};
\node (b1) at (0,-0.2) {$\scriptstyle 2$};
\draw[->,bend left] (b) to node[below] {$\scriptstyle a_2$} (a);
\draw[->,bend left] (a) to node[above] {$\scriptstyle c_2$} (b);
\end{tikzpicture}
\end{array}
\end{array}
\]
with no relations.  Thus in this example both $\CAR$ and $\CA$ are infinite dimensional, which by \ref{representability and algebra algebra properties} confirms that the contraction morphism contracts a divisor to a curve.
\end{example}

We remark that the above example, \ref{div example}, also appears in \cite[2.4]{Pagoda} and \cite[4.13]{Zerger}.

\subsection{Failure of Commutative Deformations}
Here we give two more complicated examples. The first shows that the converse to \ref{obvious}\eqref{obvious 2} is false, and the second shows that \ref{contraction theorem def functor} fails if we replace noncommutative deformations by commutative ones.  In both examples, there is only one curve in the fibre above the closed point $\m$.

\begin{example}[$\Def^{\cO_C}$ does not detect divisors]\label{trihedral1}  Consider the group $G:=A_4$ acting on its three-dimensional irreducible representation, and set $\mathfrak{R}:=\mathbb{C}[[x,y,z]]^G$.  It is well-known that in the crepant resolutions of $\Spec \mathfrak{R}$, the fibre above the origin of $\Spec \mathfrak{R}$ is one-dimensional: see \cite[\S2.4]{GNS} and \cite{NS}.  

For the crepant resolution given by $h\colon G\mbox{-Hilb}\to\Spec \mathfrak{R}$, there are three curves above the origin meeting transversally in a Type $A$ configuration.  Further, in this case the tilting bundle from $G$-Hilb has endomorphism ring isomorphic to the completion of the following McKay quiver with relations (see e.g.\ \cite[p13--14]{LS} \cite[5.2]{GLR})
\[
\End_\mathfrak{R}(\mathfrak{R}\oplus M_1\oplus M_2\oplus M_3)\cong
\begin{array}{cc}
\begin{array}{c}
\begin{tikzpicture}
\node[rotate=60,name=s,regular polygon, regular polygon sides=3, minimum size=3cm] at (0,0) {}; 
\node (R) at (s.corner 2)  {$\scriptstyle \mathfrak{R}$};
\node (1) at (s.corner 1)  {$\scriptstyle M_{1}$};
\node (2) at (s.corner 3)  {$\scriptstyle M_{3}$};
\node (M) at (s.center) {$\scriptstyle M_{2}$};
\node (Ma) at ($(s.center)+(-155:5pt)$) {};
\node (Mb) at ($(s.center)+(-25:5pt)$) {};
\draw[->]  (Ma) edge [in=180,out=-110,loop,looseness=9] node[left] {$\scriptstyle u$} (Ma);
\draw[->]  (Mb) edge [in=0,out=-70,loop,looseness=9] node[right] {$\scriptstyle v$} (Mb);
\draw[->]  (R) edge [in=-100,out=100,looseness=1] node[left] {$\scriptstyle a$} (M);
\draw[->]  (M) edge [in=80,out=-80,looseness=1] node[right] {$\scriptstyle A$} (R);
\draw[->]  (1) edge [in=160,out=-40,looseness=1] node[below] {$\scriptstyle b$} (M);
\draw[->]  (M) edge [in=-20,out=140,looseness=1] node[above] {$\scriptstyle B$} (1);
\draw[->]  (2) edge [in=40,out=-160,looseness=1] node[above] {$\scriptstyle c$} (M);
\draw[->]  (M) edge [in=-140,out=20,looseness=1] node[below] {$\scriptstyle C$} (2);
\end{tikzpicture} 
\end{array}
& 
{\small\begin{array}{c}
\begin{array}{ll}
uA=vA & au=av\\
uB=\rho vB & bu=\rho bv\\
uC=\rho^2vC & cu=\rho^2cv\\
\end{array}\\ \\
\begin{array}{l}
u^2=Aa+\rho Bb+\rho^2Cc\\
v^2=Aa+\rho^2 Bb+\rho Cc
\end{array}
\end{array}}
\end{array}
\]
where $\rho$ is a cube root of unity.  By \cite[2.15]{HomMMP}, since there are two loops on the middle vertex we see that the middle curve is a $(-3,1)$-curve, and since there are no loops on the outer vertices, the outer curves are $(-1,-1)$-curves.  Also, by inspection
\begin{equation}
\End_\mathfrak{R}(M_2)/[\mathfrak{R}\oplus M_1\oplus M_3]\cong\frac{\mathbb{C}\langle\langle u,v\rangle\rangle}{\mathrm{cl}(u^2,v^2)},
\label{zigzag algebra}
\end{equation}
where $\mathrm{cl}(u^2,v^2)$ denotes the closure of the two-sided ideal $(u^2,v^2)$.  Evidently, the above factor is infinite dimensional.  Since there is a surjective map
\[
\End_\mathfrak{R}(M_2)/[\mathfrak{R}]\twoheadrightarrow \End_\mathfrak{R}(M_2)/[\mathfrak{R}\oplus M_1\oplus M_3]
\]
it follows that $\End_\mathfrak{R}(M_2)/[\mathfrak{R}]$ must also be infinite dimensional.  

Now, contracting both the outer $(-1,-1)$-curves in $G$-Hilb we obtain a scheme $\mathfrak{U}$ and a factorization
\[
\begin{tikzpicture}
\node (A) at (0,0) {$G$-Hilb};
\node (B2) at (1.5,-0.75) {$\mathfrak{U}$};
\node (C) at (3,0) {$\Spec \mathfrak{R}$};
\draw[->] (A) -- node[below] {} (B2);
\draw[->] (B2) -- node[below] {$\scriptstyle f$} (C);
\draw[->] (A) -- node[above] {$\scriptstyle h$}(C);
\end{tikzpicture}
\]
The example we consider is $f\colon \mathfrak{U}\to \Spec \mathfrak{R}$. By construction, there is only one curve above the origin.  As in \cite[4.6]{KIWY} 
\[
\Db(\coh \mathfrak{U})\cong\Db(\mod \End_\mathfrak{R}(\mathfrak{R}\oplus M_2)).
\]
Set $\AB:=\End_\mathfrak{R}(\mathfrak{R}\oplus M_2)$, then the quiver for $\AB$ is obtained from the above McKay quiver by composing two-cycles that pass through $M_1$ and $M_3$.  In particular, in the quiver for $\AB$ there is no loop at the vertex corresponding to $\mathfrak{R}$, so $\AB_{\fib}=\mathbb{C}$.  In particular, by \ref{oCdefmain}, $\Def^{\cO_C}$ is representable.

On the other hand $\CA=\End_\mathfrak{R}(M_2)/[\mathfrak{R}]$, and we have already observed that this is infinite dimensional, so  by \ref{contraction theorem} $f$ contracts a divisor to a curve. It is also possible to observe this divisorial contraction using the explicit calculations of open covers in \cite{NS}.
\end{example}

\begin{example}[$\cDef^{\cJ}$ does not detect divisors]\label{trihedral2} Consider again $\mathfrak{R}:=\mathbb{C}[[x,y,z]]^G$ where $G$ is the alternating group above in \ref{trihedral1}.  Now, contracting instead the middle curve in $G$-Hilb (instead of the outer curves we contracted above) gives a factorization
\[
\begin{tikzpicture}
\node (A) at (0,0) {$G$-Hilb};
\node (B2) at (1.5,-0.75) {$W$};
\node (C) at (3,0) {$\Spec \mathfrak{R}$};
\draw[->] (A) -- node[below] {$\scriptstyle g$} (B2);
\draw[->] (B2) -- node[below] {$\scriptstyle $} (C);
\draw[->] (A) -- node[above] {$\scriptstyle h$}(C);
\end{tikzpicture}
\]
The original middle curve contracts to a closed point $\m$ in $W$, so picking an affine open $\Spec T$ in $W$ containing $\m$, and passing to the formal fibre of $g$ over this point, we obtain a morphism 
\[
g\colon \mathfrak{W}\to \Spec \mathfrak{T}.
\]
To this contraction, we associate the contraction algebra $\CA$ using the procedure in \ref{def basic algebra}. By \cite[3.5(2)]{HomMMP}, it follows from the uniqueness of prorepresenting object that
\[
\CA\cong \End_{\mathfrak{R}}(M_2)/[\mathfrak{R}\oplus M_1\oplus M_3],
\]
where we have recycled notation from \ref{trihedral1}.  Hence by \eqref{zigzag algebra} we see that
\[
\CA\cong \frac{\mathbb{C}\langle\langle u,v\rangle\rangle}{\mathrm{cl}(u^2,v^2)}.
\]
We have already observed that this is infinite dimensional, so $g$ contracts a divisor by \ref{contraction theorem}.  Alternatively, we can see that $g$ contracts a divisor by using the explicit open cover as in \cite[p40]{NS}.

On the other hand, by general theory (see e.g.\ \cite[3.2]{DW1})  $\cDef^\cJ$ is prorepresented by the abelianization of $\CA$, which in this case is simply
\[
\CA^{\ab}\cong \frac{\mathbb{C}[[u,v]]}{(u^2,v^2)}.
\]
By inspection, this is finite dimensional.  This shows that the commutative deformation functor $\cDef^\cJ$ is representable, even although a divisor is contracted to a curve. 
\end{example}

\end{document}